\documentclass[reqno, amsthm, 10pt, a4paper]{amsart}
\usepackage{tikz}
\usetikzlibrary{shapes,arrows}
\usepackage[small]{caption}
\usepackage{mathrsfs,amssymb}

\usepackage{amsfonts}
\usepackage{amsmath}
\usepackage{amssymb}
\usepackage{latexsym}
\usepackage{mathrsfs}
\usepackage{graphicx}

\usepackage{epic,eepic,ecltree} 

\input xy
\xyoption{all}

\usepackage{xy} 

\newcommand{\dedge}[1]{\ar@{--}[#1]}
\newcommand{\edge}[1]{\ar@{-}[#1]}
\newcommand{\lulab}[1]{\ar@{}[l]_<<{#1}}
\newcommand{\rulab}[1]{\ar@{}[r]^<<{#1}}
\newcommand{\ldlab}[1]{\ar@{}[l]^<<{#1}}
\newcommand{\rdlab}[1]{\ar@{}[r]_<<{#1}}
\newcommand{\node}{*+[o][F-]{ }}

 \newcommand{\gh}{\mathscr{H}}
  \newcommand{\gr}{\mathscr{R}}
  \newcommand{\gl}{\mathscr{L}}
  \newcommand{\gp}{\mathscr{P}}

  \newcommand{\gq}{\mathscr{Q}}
  \newcommand{\gG}{\mathscr{G}}

 \newcommand{\bba}{\mathbb{A}}
 \newcommand{\bbb}{\mathbb{B}}
 \newcommand{\bbe}{\mathbb{E}}
 \newcommand{\bbp}{\mathbb{P}}
 \newcommand{\bbq}{\mathbb{Q}}

 \newcommand{\bbf}{\mathbb{F}}

 \newcommand{\bbh}{\mathbb{H}}
 \newcommand{\bbd}{\mathbb{D}}
 \newcommand{\bby}{\mathbb{Y}}
 \newcommand{\da}{\downarrow}
 \newcommand{\ua}{\uparrow}

\newtheorem{thm}{Theorem}

\newtheorem{cor}{Corollary}
\newtheorem{defn}{Definition}
\newtheorem{prop}{Proposition}
\newtheorem{lem}{Lemma}

\theoremstyle{definition}
\newtheorem{remark}{Remark}

\newcommand{\lb}{\langle}
\newcommand{\rb}{\rangle}
\usepackage[active]{srcltx}

\usepackage{verbatim}
\usepackage{enumerate}

\begin{document}

\title[]{Homotopy bases and finite derivation type for Sch\"{u}tzenberger groups of monoids}

\maketitle

\vspace{-6mm}

\begin{center}

    R. GRAY\footnote{
	  Part of this work was done while this author held an EPSRC Postdoctoral Fellowship at the University of St Andrews.}

    \medskip

    Centro de \'{A}lgebra da Universidade de Lisboa, \\ Av. Prof. Gama Pinto, 2,  1649-003 Lisboa,  \ Portugal.

    \medskip

    \texttt{rdgray@fc.ul.pt}

    \bigskip

    A. MALHEIRO\footnote{
	This work was developed within the projects POCTI-ISFL-1-143 and PTDC/MAT/69514/2006 of CAUL, financed by FCT and FEDER, and
	Supported by the Treaty of Windsor scheme of the British Council of Great Britain and Portugal. }

    \medskip

    Centro de \'{A}lgebra da Universidade de Lisboa, \\ Av. Prof. Gama Pinto, 2,  1649-003 Lisboa,  \ Portugal.
    \\ and \\ Departamento de Matem\'{a}tica, Faculdade de Ci\^{e}ncias e Tecnologia \\ da Universidade Nova de Lisboa,
    2829-516 Caparica, Portugal

    \medskip

    \texttt{malheiro@cii.fc.ul.pt}

    \bigskip

    S. J. PRIDE

    \medskip

    Department of Mathematics,\ University of Glasgow,  \\
    University Gardens, G12 8QW, Scotland.

    \medskip

    \texttt{stephen.pride@gla.ac.uk} \\
\end{center}

\begin{abstract}
Given a finitely presented monoid and a homotopy base for the monoid, and given an arbitrary Sch\"{u}tzenberger group of the monoid, the main result of this paper gives a homotopy base, and presentation, for the Sch\"{u}tzenberger group. In the case that the $\gr$-class $R$ of the Sch\"{u}tzenberger group $\gG(H)$ has only finitely many $\gh$-classes, and there is an element $s$ of the multiplicative right pointwise stabilizer of $H$, such that under the left action of the monoid on its $\gr$-classes the intersection of the orbit of the $\gr$-class of $s$ with the inverse orbit of $R$ is finite, then finiteness of the presentation and of the homotopy base is preserved.

\

\end{abstract}

\section{Introduction}
\label{sec_intro}






The history of rewriting systems in monoids and groups is about one hundred years old. 
A presentation for a monoid or group may be viewed as a rewriting system. Sometimes this system may be chosen to have certain desirable properties which can then lead to solutions to classical algorithmic problems.  
For example, if the rewriting system is finite and complete (also called convergent) then it can be used to solve the word problem and
to find normal forms for elements of the monoid; see \cite{Book1}. 
This is one of the main reasons for the importance of complete rewriting systems in group and semigroup theory, and there is a vast body of literature on the subject \cite{Silva2009_2, Silva2009_1, Chouraqui2006, Miasnikov2009, Goodman2008, Hermiller1999, Pride2005}.

Of course, not every finitely presented monoid admits a presentation by a finite complete rewriting system (even if the monoid has soluble word problem). One of the key techniques for showing that a given monoid or group does \emph{not} admit such a presentation is to show that the monoid does not satisfy one of a family of geometric finiteness properties which are known to be satisfied by monoids that admit finite complete rewriting systems. 
We refer the reader to the survey articles \cite{Cohen1997} and \cite{OttoSurvey} for more about finiteness properties of this kind. 
In this article our interest is in one such property, known as finite derivation type; originally introduced and investigated by Squier in \cite{Squier1}. 

Given a rewriting system (i.e. monoid presentation) $\lb A | \mathfrak{R} \rb$ one builds a (combinatorial) $2$-complex $\mathcal{D}$, called the \emph{Squier complex}, whose $1$-skeleton has vertex set $A^*$ and edges corresponding to applications of relations from $\mathfrak{R}$, and that has $2$-cells adjoined for each instance of ``non-overlapping'' applications of relations from $\mathfrak{R}$ (see Section~\ref{sec_prelims} for a detailed definition). There is a natural action of the free monoid $A^*$ on the Squier complex $\mathcal{D}$. A  collection of closed paths in $\mathcal{D}$ is called a \emph{homotopy base} if the complex obtained by adjoining cells for each of these paths, and those that they generate under the action of the free monoid on the Squier complex, has trivial fundamental groups. A monoid defined by a presentation is said to have \emph{finite derivation type} (FDT for short) if the corresponding Squier complex admits a finite homotopy base. Squier \cite{Squier1} proved that the property FDT is independent of the choice of finite presentation, so we may speak of FDT monoids. The original motivation for studying this notion is Squier's result \cite{Squier1} which says that if a monoid admits a presentation by a finite complete rewriting system then the monoid must have finite derivation type. Further motivation for the study of these concepts comes from the fact that the fundamental groups of connected components of Squier complexes, which are called \emph{diagram groups}, have turned out to be a very interesting class of groups, and have been extensively studied in
\cite{Guba2006, Farley2003, Farley2005, Guba, Guba1999, GubaSapir2006}.

Associated in a natural way with an arbitrary monoid is a certain family of groups, called the Sch\"{u}tzenberger groups of the monoid (introduced in \cite{Schutzenberger1957,Schutzenberger1958}) which encode much of the underlying group theoretic structure of the monoid.  
Our interest here is in relating the property FDT holding in the monoid to the same property holding in its Sch\"{u}tzenberger groups, in particular identifying circumstances under which FDT is inherited by the Sch\"{u}tzenberger group from the monoid. 
Corresponding results for finite generation and presentability were obtained in \cite{Ruskuc2000}. 
For finitely presented groups FDT is equivalent to the homological finiteness property ${\rm FP}\sb 3$ (see \cite{Cremanns1}). Therefore,
as well as giving a natural higher dimension analogue of the main result of \cite{Ruskuc2000}, the main result we present here also has the potential to be applied to show that particular monoids do not have FDT, and therefore do not admit presentations by finite complete rewriting systems, by showing that an appropriately ``embedded'' Sch\"{u}tzenberger group is not of type ${\rm FP}\sb 3$.

Including this introduction, this article comprises five sections.
In Section~2 we recall some basic ideas about semigroup actions, including the definition of Sch\"{u}tzenberger group, and then state our main result and explain some of its consequences. Basic definitions and results about rewriting systems, homotopy bases, and finite derivation type will be given in Section~3, where we also recall some fundamental ideas from the structure theory of semigroups that we shall need. In Section~4 we give a presentation for an arbitrary Sch\"{u}tzenberger group and we prove the first part of our main result Theorem~\ref{thm_main_result}(i). 
Finally in Section~5 we give a homotopy base for an arbitrary Sch\"{u}tzenberger group, and prove Theorem~\ref{thm_main_result}(ii) which is the main result of this article.

\section{Statement of main result}
\label{sec_main_result}

Many of the most powerful tools that are available for investigating the algebraic structure of semigroups are best understood from the point of view of actions. The most basic observation of all, of course, is that a monoid $S$ acts on itself via left (and dually right) multiplication. Arising from any action of a monoid $S$ on a set $X$ is a preorder structure on $X$ given by setting $x \leq y$ if and only if $x$ belongs to the orbit $Sy$ of $y$. The classes of the corresponding equivalence relation are the strong orbits of the action, which of course partition the set $X$. By the inverse orbit of $x \in X$ we mean the set of all $y \in X$ such that $x$ belongs to the orbit $Sy$ of $y$. In the case of the natural actions of $S$ on itself via left multiplication the orbits are the principal left ideals, the preorder is the $\leq_\gl$ relation,
$
s \leq_\gl t$ if and only if  $s \in St,
$
and the strong orbits of this action are nothing but the Green's $\gl$-classes of the monoid (in the sense of \cite{Green1} or \cite[Chapter~2]{Howie1}). Dually the right multiplicative action gives rise to the principal right ideals, the preorder $\leq_\gr$ and the $\gr$-relation. The next thing one observes is that the left multiplicative action induces an action $*$ of $S$ on the set $S / \gr$ of $\gr$-classes where
\[
s * R = R' \Leftrightarrow sR \subseteq R'
\quad
(s \in S, R, R' \in S / \gr).
\]
This is well-defined as $\gr$ is a left congruence (i.e. $x \gr y$ implies $sx \gr sy$ for all $x,y,s \in S$). Dually $S$ acts on the set of $\gl$-classes $S / \gl$ on the right.

Intersecting $\gr$ with $\gl$ we obtain Green's relation $\gh = \gr \cap \gl$. The importance of this relation lies in the way that the $\gh$-classes reflect the underlying local group-theoretic structure of the semigroup. The first place this idea is encountered is when one considers the subgroups of a semigroup.
Let $e$ be an idempotent of a semigroup $S$. The set $eSe$ is a submonoid in $S$ and is the largest
submonoid (with respect to inclusion) whose identity element is $e$. The group of units $G_e$ of $eSe$, that
is the group of elements of $eSe$ that have two sided inverses with respect to $e$, is the largest subgroup
of S (with respect to inclusion) whose identity is e and is called the maximal subgroup of $S$ containing $e$.
If an $\gh$-class $H$ contains an idempotent element $e$ then $H$ must be a maximal subgroup of $S$ (with identity element $e$). Conversely, every maximal subgroup arises in this way. Thus maximal subgroups and group $\gh$-classes are one and the same. On the other hand, if $H$ is an $\gh$-class that does not contain an idempotent then, on the surface at least, $H$ is as far from being a group as one could imagine in the sense that
$
H^2 \cap H = \varnothing.
$
However, Sch\"{u}tzenberger \cite{Schutzenberger1957, Schutzenberger1958} observed 
that given an arbitrary $\gh$-class of a monoid $S$ there is a natural way that one may associate a group $\gG(H)$ with $H$. Moreover, this is done in such a way that when $H$ is a maximal subgroup (i.e. does contain an idempotent) we have $\gG(H) \cong H$. The (right) Sch\"{u}tzenberger group $\gG(H)$ of $H$ is obtained by taking the action of the setwise stabilizer of $H$ on $H$, under right multiplication by elements of the monoid, and making it faithful. That is, given an arbitrary $\gh$-class $H$ of $S$, let $\mathrm{Stab}(H) = \{ s \in S : Hs=H \}$ denote the \emph{(right) stabilizer} of $H$ in $S$. Then define an equivalence $\sigma=\sigma(H)$ on the stabilizer by $(x,y) \in \sigma$ if and only if $hx = hy$ for all $h\in H$. It is straightforward to verify that $\sigma$ is a congruence, and that $\gG(H)=\mathrm{Stab}(H) / \sigma$ is a group, which is called the \emph{right Sch\"{u}tzenberger group} of $H$. There is a dual notion of left Sch\"{u}tzenberger group which turns out to be naturally isomorphic to the right Sch\"{u}tzenberger group.

So, the natural actions of a semigroup on itself, and its subsets, give rise to a partition $S / \gh$ of $S$ into $\gh$-classes, with each of which 
is a naturally associated Sch\"{u}tzenberger group 
$\gG(H)$. From this point of view one may view an arbitrary semigroup $S$ as being built up from its collection of Sch\"{u}tzenberger groups $\{ \mathcal{G}(H) : H \in S / \gh \}$. This viewpoint provides one of the most powerful routes connecting semigroup theory with group theory, and naturally leads one to consider to what extent a semigroup $S$ may be understood via analysis of its Sch\"{u}tzenberger groups.

An important special case of this approach may be found in the world of regular semigroups. A semigroup $S$ is called (von Neumann) regular if for every $x \in S$ there exists $x' \in S$ satisfying $xx'x=x$. This is equivalent to saying that every $\gr$-class of $S$ contains an idempotent. Thus, for regular semigroups the general approach outlined above reduces to the question of the extent to which regular semigroups may be studied via their maximal subgroups. Deep insights into regular semigroups can often be obtained by careful analysis of their maximal subgroups; see \cite{Brittenham2009} for a recent example of 
a result of this kind. 
Correspondingly, insights into arbitrary semigroups can be obtained by studying their Sch\"{u}tzenberger groups.  

Our interest here is in homotopy bases, and finite derivation type, and the extent to which the study of these concepts for a monoid relates to the study of the corresponding properties in the Sch\"{u}tzenberger groups of the monoid. Specifically, given a finitely presented monoid $S$ and a homotopy base for the monoid, and given an arbitrary Sch\"{u}tzenberger group of the monoid, we shall (in Section~\ref{sec_hbase}) construct a homotopy base, and presentation (in Section~\ref{sec_presentation}), for the Sch\"{u}tzenberger group. As a corollary, we obtain the following which is the main result of the article.

\begin{thm}
\label{thm_main_result}
Let $S$ be a monoid, $H$ be an $\gh$-class of $S$ with Sch\"{u}tzenberger group $\gG$, and let $R$ be the $\gr$-class that contains $H$. Furthermore let $s \in S$ be any element from the right multiplicative pointwise stabilizer of $H$. If:
\begin{enumerate}
\item[\emph{(a)}] $R$ has only finitely many $\gh$-classes; and
\item[\emph{(b)}] under the left action of $S$ on its $\gr$-classes, the intersection of the orbit of $R_s$ with the inverse orbit of $R$ is finite,
\end{enumerate}
then we have the following:
\begin{enumerate}
\item[\emph{(i)}] If $S$ if finitely presented then $\gG$ is finitely presented.
\item[\emph{(ii)}] If $S$ has finite derivation type then $\gG$ has finite derivation type.
\end{enumerate}
\end{thm}

Note that in the above statement such an element $s$ always exists, since we can always take $s=1$ (although of course we cannot guarantee that (b) is satisfied).  
We note that the result does not hold if either condition (a) or condition (b) is dropped. Indeed, if we keep condition (a) but drop condition (b) then a counterexample may be found in \cite[Section~6]{Ruskuc2000}. 
On the other hand, if we keep condition (b) and drop condition (a) then \cite[Proposition~3.3]{Ruskuc1999} serves as a counterexample. 

Let us make a few further remarks about Theorem~\ref{thm_main_result}.

\smallskip

\noindent $\bullet$ Our main result is part (ii) of Theorem~\ref{thm_main_result}. Part (i) is new, and it was necessary for us to establish (i) before proving (ii). We obtain (i) by carefully modifying the proof of a less general result proved in \cite{Ruskuc2000}. One benefit of working under this weaker set of hypotheses is that Theorem~\ref{thm_main_result} now provides a common generalisation of the main results of \cite{Ruskuc1999} and \cite{Ruskuc2000}, as we now explain.
	\begin{itemize}
	\item[--] If $H$ contains an idempotent $f$ (i.e. is a subgroup of $S$) then $f$ is clearly in the right pointwise stabilizer of $H$
	and so we may apply Theorem~\ref{thm_main_result} setting $s = f$. But then $R = R_f = R_s$ and so condition (b) automatically holds (the set in question having size $1$) and so parts (i) and (ii) of the theorem hold just under the assumption (a). Thus we obtain the results \cite[Corollary~3]{Ruskuc1999} and \cite[Theorem~2.9]{GrayMalheiro} as a corollary of Theorem~\ref{thm_main_result}.
        \end{itemize}	
        \begin{itemize}
	\item[--] By setting $s = 1_S$, the identity of the monoid $S$ which certainly stabilizes $H$ pointwise on the right,
	we obtain precisely the main result \cite[Corollary~4.3]{Ruskuc2000} as a corollary of part (i) of Theorem~\ref{thm_main_result}.   
	\end{itemize}

\smallskip



\noindent $\bullet$ The properties of being finitely generated, presented, and having finite derivation type may, respectively, naturally be viewed as being zero-, one-, and two-dimensional notions. This suggests that Theorem~\ref{thm_main_result} should extend to arbitrary dimensions, it is just not quite clear at the present time what the precise formulation of such a result should be. We note that the homological finiteness property $\mathrm{FP}_n$ is certainly not the right candidate for this since Theorem~\ref{thm_main_result} is far from being true for the property $\mathrm{FP}_n$; see \cite{Kobayashi2009, GrayPride}.

\smallskip

\noindent $\bullet$ We prove Theorem~\ref{thm_main_result} below using combinatorial rewriting methods which allow us to give an explicit presentation and homotopy base for the Sch\"{u}tzenberger group. It would be interesting to know whether it might be possible to prove Theorem~\ref{thm_main_result} using a topological approach instead. Currently the most success in this direction may be found in results about presentations given in \cite{Steinberg1} and \cite{Steinberg2010}. Unfortunately the results that may be proved using these topological methods still fall very far short of the level of generality achieved by Theorem~\ref{thm_main_result} above (and in fact do not manage to encompass the main result of \cite{Ruskuc2000}).

\smallskip

As mentioned above, for finitely presented groups, finite derivation type is equivalent to the homological finiteness property $\mathrm{FP}_3$. Thus one consequence of the above theorem is that if a monoid is to admit a presentation by a finite complete rewriting system, and therefore have $\mathrm{FDT}$, then in many situations this will influence the homological finiteness properties of the Sch\"{u}tzenberger groups of that monoid. Thus, the result has the potential to be used as a tool to show that particular monoids cannot be presented by finite complete rewriting system, by analyzing the homological properties of their Sch\"{u}tzenberger groups.

One very particular situation where the hypotheses (a) and (b) of Theorem~\ref{thm_main_result} are satisfied is when $S$ has finitely many left and right ideals. In this situation, the converse of  Theorem~\ref{thm_main_result}(i) for finite presentability is known to be true: in \cite{Ruskuc2000} it was shown that a monoid with finitely many left and right ideals is finitely presented if and only if all of its Sch\"{u}tzenberger groups are. It is natural to ask whether the same result holds with $\mathrm{FDT}$ in place of finite presentability. 
In one direction, passing from the monoid to the Sch\"{u}tzenberger group, we see from Theorem~\ref{thm_main_result} that the result does hold.
\begin{cor}
\label{corol_finite}
Let $S$ be a monoid with finitely many left and right ideals. If $S$ has finite derivation type then all the Sch\"{u}tzenberger groups of $S$ have finite derivation type. 
\end{cor}
For von Neumann regular semigroups the converse of Corollary~\ref{corol_finite} also holds: in \cite{GrayMalheiro} it was shown that a regular semigroup with finitely many left and right ideals has finite derivation type if and only if all of its Sch\"{u}tzenberger groups (equivalently, all of its maximal subgroups) do. Rather surprisingly however, this result does \emph{not} extend to non-regular semigroups (like it does for finite presentability). Indeed, in \cite{GrayPrideMalheiro} we give an example of a monoid with finitely many left and right ideals, all of whose Sch\"{u}tzenberger groups have $\mathrm{FDT}$, but such that the monoid itself does not have $\mathrm{FDT}$.

\section{Preliminaries}
\label{sec_prelims}

\subsection*{Presentations and rewriting systems}

Let $A$ be a non-empty set, that we call an \emph{alphabet}. We use $A^*$ to denote the free monoid over $A$, and $A^+$ to denote the free semigroup.  
A \emph{rewriting system} over $A$ is a subset $\mathfrak{R} \subseteq A^* \times A^*$. An element of $\mathfrak{R}$ is called a \emph{rule}, and we often write $r_{+1}=r_{-1}$ for $(r_{+1}, r_{-1}) \in \mathfrak{R}$. For $u, v \in A^*$ we write $u \rightarrow_\mathfrak{R} v$ if $u \equiv w_1 r_{+1} w_2$, and $v \equiv w_1 r_{-1} w_2$ where $(r_{+1},r_{-1}) \in \mathfrak{R}$ and $w_1, w_2 \in A^*$. The reflexive symmetric transitive closure $\leftrightarrow_\mathfrak{R}^{*}$ of $\rightarrow_\mathfrak{R}$ is precisely the congruence on $A^*$ generated by $\mathfrak{R}$. The ordered pair $\langle A | \mathfrak{R} \rangle$ is called a \emph{monoid presentation} with \emph{generators} $A$ and set of \emph{defining relations} $\mathfrak{R}$. If $S$ is a monoid that is isomorphic to $A^*  / \leftrightarrow_\mathfrak{R}^{*}$ we say that $S$ is the \emph{monoid defined by the presentation $\lb A | \mathfrak{R} \rb$}. 
Thus every word $w \in A^*$ represents an element of $S$. 
For two words
$w_1$, $w_2 \in A^*$ we write $w_1 \equiv  w_2$ if they are identical as words, and $w_1 = w_2$ if they
represent the same element of $S$, i.e., if $w_1 / \leftrightarrow_\mathfrak{R}^{*} = w_2 / \leftrightarrow_\mathfrak{R}^{*}$. A monoid is said to be \emph{finitely presented} if it may be defined by a presentation with finitely many generators and a finite number of defining relations. For $w \in A^*$ we write $|w|$ to denote the total number of letters in $w$, which we call the \emph{length} of the word $w$. 

\subsection*{Derivation graphs, homotopy bases, and finite derivation type}

We adopt a combinatorial model of $2$-complexes. A \emph{graph} (in the sense of Serre \cite{SerreTrees}) consists of a set $V$ of vertices, a set $E$ of edges and three functions
\[
\iota: E \rightarrow V, \quad
\tau: E \rightarrow V, \quad
{}^{-1}:E \rightarrow E,
\]
called initial, terminal and inverse (respectively),
satisfying $\iota e^{-1} = \tau e $, $\tau e^{-1} = \iota e$, $(e^{-1})^{-1} = e$ and $e^{-1} \neq e$ for all $e \in E$. A \emph{non-empty path} $p$ will consist of a sequence of edges $e_1 \circ e_2 \circ \ldots\circ e_n$ with $\tau e_i = \iota e_{i+1}$ $(1 \leq i < n)$. We then define
\[
\iota p = \iota e_1, \quad
\tau p = \tau e_n, \quad
p^{-1} = e_n^{-1} \circ \ldots\circ e_2^{-1}\circ e_1^{-1}.
\]
For each vertex $v$ there is also the \emph{empty path} $1_v$ at $v$ with no edges and $\iota(1_v) = \tau(1_v) = v$ and $1_v^{-1} = 1_v$. A \emph{closed path}, based at $v$, is a path $p$ satisfying 
$\iota p = \tau p = v$. If $p$ and $q$ are paths with $\tau p = \iota q$ then we can form the \emph{product path} $p\circ q$ consisting of the edges of $p$ followed by the edges of $q$.

A \emph{$2$-complex} $\mathcal{K}$ will consist of a graph, together with a set $C$ of closed paths (the \emph{defining paths}) in the graph. The underlying graph of $\mathcal{K}$ will be called the \emph{$1$-skeleton} of $\mathcal{K}$. We define elementary operations (homotopy moves) on paths as follows:
\begin{enumerate}
\item[(i)] Insert/delete a subpath $e\circ e^{-1}$ ($e$ an edge);
\item[(ii)] Replace a subpath $p$ by $q$ if $p\circ q^{-1}$ is a cyclic permutation of some path in $C \cup C^{-1}$.
\end{enumerate}
A path will be said to be \emph{null-homotopic} if it is homotopic to an empty path.

With a monoid presentation $\gp= \lb A | \mathfrak{R} \rb$ we associate a $2$-complex $\mathcal{D} = \mathcal{D}(\gp)$, which is called the \emph{Squier complex} of $\gp$, as follows. The $1$-skeleton, which we denote by $\Gamma(\gp)$, has vertex set $A^*$, and edge set consisting of the quadruples
\[
\bbe = (\alpha, r, \epsilon, \beta), \quad \alpha, \beta \in A^*, \; r \in \mathfrak{R}, \; \epsilon = \pm 1.
\]
The initial, terminal and inverse functions given by
\[
\iota \bbe \equiv \alpha r_\epsilon \beta, \quad
\tau \bbe \equiv \alpha r_{-\epsilon} \beta, \quad
\bbe^{-1} = (\alpha, r, -\epsilon, \beta),
\]
respectively. The edge $\bbe$ is called \emph{positive} if $\epsilon = +1$. We may represent the edge $\bbe=(\alpha, r, \epsilon, \beta)$ geometrically by an object called a \emph{monoid picture} as follows:
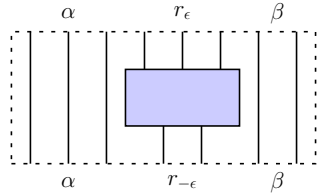
\begin{center}
\scalebox{0.5}
{
\begin{tikzpicture}
[very thick,
blackbox/.style={draw, fill=blue!20, rectangle, minimum height=1.5cm, minimum width=3cm},
dottedbox/.style={draw, rectangle, loosely dashed, minimum height=3.5cm, minimum width=8cm}]
\node at (4.5,2.75) [dottedbox] {};
\node at (5,2.75) [blackbox] {};
\draw (1,1) -- (1,4.5);
\draw (2,1) -- (2,4.5);
\draw (3,1) -- (3,4.5);
\draw (7,1) -- (7,4.5);
\draw (8,1) -- (8,4.5);
\draw (4,3.5) -- (4,4.5);
\draw (5,3.5) -- (5,4.5);
\draw (6,3.5) -- (6,4.5);
\draw (4.5,1) -- (4.5,2);
\draw (5.5,1) -- (5.5,2);
\node at (2,5) {\huge $\alpha$};
\node at (2,0.5) {\huge $\alpha$};
\node at (5,5) {\huge $r_\epsilon$};
\node at (5,0.5) {\huge $r_{-\epsilon}$};
\node at (7.5,5) {\huge $\beta$};
\node at (7.5,0.5) {\huge $\beta$};
\end{tikzpicture}
}
\end{center}
Following \cite{Guba}, the rectangle in the centre of the picture is called a \emph{transistor}, and corresponds to the relation $r$ from the presentation, while the line segments in the digram are called \emph{wires}, with each wire labelled by a unique letter from the free monoid $A^*$.  
The monoid picture for the inverse $\bbe^{-1}$ of an edge is obtained by taking the vertical mirror image of the picture of $\bbe$. By ``stacking'' such pictures one on top of the other, and joining corresponding wires, we obtain pictures for arbitrary paths in the graph $\Gamma(\gp)$. 
See \cite{Guba} for a more formal and comprehensive treatment of monoid pictures. Specifically, see \cite[Definition~4.1]{Guba} for a formal definition of monoid picture. A formal treatment of pictures will not be required since here they will only be used to illustrate the geometry behind the arguments. 

The free monoid $A^*$ acts on $\Gamma(\gp)$ in a natural way via left and right multiplication, where for the edge $\bbe$ above, and any two words $\gamma, \delta \in A^*$ we define
\[
\gamma \cdot \bbe \cdot \delta = (\gamma\alpha, r, \epsilon, \beta\delta).
\]
These actions extend to paths in the obvious way. In terms of monoid pictures acting on the left by $\gamma$ and the right by $\delta$ corresponds to inserting a series of vertical wires to the left of the picture, corresponding to the letters of $\gamma$, and similarly on the right for $\delta$.

For every $r \in {\mathfrak{R}}$ and $\epsilon = \pm 1$ define
$
\bba_r^\epsilon = (1, r, \epsilon, 1).
$
We call such edges \emph{elementary}. Clearly elementary edges are in obvious correspondence with the relations of the presentation, and every edge of
$\Gamma(\gp)$ can be written uniquely in the form $\alpha \cdot \bba \cdot \beta$ where $\alpha, \beta \in A^*$ and $\bba$ is elementary.

Note that an edge in $\Gamma(\gp)$ corresponds to a single application of a rewriting rule, and so a path $\bbp$ corresponds to a deduction that the words $\iota \bbp$ and $\tau \bbp$ represent the same element of the monoid defined by the presentation. Thus $\Gamma(\gp)$ is called the \emph{derivation graph} of the presentation. Of course, there is a one-to-one correspondence between the elements of the monoid $S$ and the connected components of the derivation graph.

\begin{sloppypar}
To obtain the Squier complex $\mathcal{D}$ the idea is that we want to adjoin defining paths which correspond to non-overlapping applications of relations from $\mathfrak{R}$. That is, it may be that a word $w$ has two disjoint occurrences of rewriting rules in the sense that
$
w \equiv \alpha r_\epsilon \beta \alpha' r_{\epsilon'}' \beta'
$
where $\alpha, \beta, \alpha', \beta' \in A^*$, $r, r' \in \mathfrak{R}$ and $\epsilon, \epsilon' \in \{ -1, +1 \}$. Let $\bbe = (\alpha, r, \epsilon, \beta)$ and $\bbe' = (\alpha', r', \epsilon', \beta')$. Then the paths
\[
\bbp = (\bbe \cdot \iota \bbe')\circ (\tau \bbe \cdot \bbe'),
\quad
\bbp' = (\iota \bbe \cdot \bbe')\circ (\bbe \cdot \tau \bbe')
\]
give two different ways of rewriting the word $w \equiv \alpha r_\epsilon \beta \alpha' r_{\epsilon'}' \beta'$ to the word $w \equiv \alpha r_{-\epsilon} \beta \alpha' r_{-\epsilon'}' \beta'$, where in $\bbp$ we first apply the left hand relation and then the right hand, while in $\bbp'$ the relations are applied in the opposite order. In this situation we will say that the two edges in the path $\bbp$ are disjoint, and similarly for $\bbp'$. We want to regard these two paths as being essentially the same, and we  achieve this by adjoining the closed path $\bbp \bbp'^{-1}$ as a defining path. In terms of pictures we see that $\bbp$ may be transformed into $\bbp'$ by ``pushing down'' the transistor corresponding to $r$ on the left and ``pulling up'' the $r'$ transistor on the right. The intuitive idea is that since the rules do not overlap this may be done without the transistors, or wires, getting in each other's way. 
\end{sloppypar}

Thus, for any two edges $\bbe_1, \bbe_2$ in the graph we adjoin the defining path
\begin{equation}
\label{eqn_pullup_pushdown}
[\bbe_1,\bbe_2] =
(\bbe_1 \cdot \iota \bbe_2)\circ
(\tau \bbe_1 \cdot \bbe_2)\circ
(\bbe_1^{-1} \cdot \tau \bbe_2)\circ
(\iota \bbe_1 \cdot \bbe_2^{-1}).
\end{equation}
The left and right actions of $A^*$ on the graph $\Gamma(\gp)$ extend naturally to actions on the complex $\mathcal{D}$ via
\[
\alpha \cdot [\bbe_1, \bbe_2] \cdot \beta
=
[\alpha \cdot \bbe_1, \bbe_2 \cdot \beta] \quad (\alpha, \beta \in A^*).
\]

Now, given any set $X$ of closed paths in $\mathcal{D}$, we can form a new $2$-complex $\mathcal{D}^X$ by adjoining the additional defining paths $\alpha \cdot \bbp \cdot \beta, (\alpha, \beta \in A^*, \bbp \in X)$. We say that $X$ is a \emph{homotopy base} (or \emph{homotopy trivializer}) if all closed paths in $\mathcal{D}^X$ are null-homotopic. If two paths $\bbp, \bbq$ are homotopic in $\mathcal{D}$ we write $\bbp \sim_0 \bbq$, and we write $\bbp \sim_X \bbq$ if they are homotopic in $\mathcal{D}^X$.

\begin{defn}
\label{def_FDT}
We say that the monoid presentation $\gp$ has \emph{finite derivation type} (written $\mathrm{FDT}$ for short) if its Squier complex admits a finite homotopy base.
\end{defn}

The definition of $\mathrm{FDT}$ given above was first introduced for monoids by Squier in \cite{Squier1}. It was shown by Squier that if two finite presentations $\gp_1$ and $\gp_2$ define the same monoid (i.e. they are Tietze equivalent) then $\gp_1$ has FDT if and only if $\gp_2$ has $\mathrm{FDT}$. This allows us to talk of FDT monoids.

The following easy lemma will prove useful.

\begin{lem}[{\cite[Lemma~2.1]{Kobayashi2000}}]
\label{lem_homgen}
A set $X$ of closed paths in $\Gamma = \Gamma(\lb A | \mathfrak{R} \rb)$ is a homotopy base if and only if for any closed path $\bbp$ in $\Gamma$, there are $v_i, w_i \in A^*$, paths $\bbp_i$  in $\Gamma$ and $\bbq_i \in X \cup X^{-1}$, $i=1,\ldots,n$, $n \geq 0$, such that:
\[
\bbp
\sim_0
\bbp_1^{-1} \circ (v_1\cdot \bbq_1\cdot w_1) \circ \bbp_1 \circ \cdots \circ \bbp_n^{-1} \circ (v_n\cdot \bbq_n\cdot w_n) \circ \bbp_n.
\]
\end{lem}

The rest of this section is spent recalling some fundamental ideas from the structure theory of semigroups. For more details we refer the reader to \cite{Howie1}, or more recently \cite[Appendix~A]{SteinbergBook2009}.

\subsection*{Green's relations and Sch\"{u}tzenberger groups}

One obtains significant information about a semigroup by considering its ideal structure. Since their introduction in \cite{Green1}, Green's relations have
become a fundamental tool for describing the ideal structure of monoids and semigroups. 
The following result, the first part of which is usually referred to as Green's lemma, lists some well-known basic properties of $\gr$, $\gl$ and $\gh$ (see Section~\ref{sec_main_result} above for the definitions of the relations $\gr$, $\gl$ and $\gh$). For proofs of these basic facts we refer the reader to \cite[Chapter~2]{Howie1}.

\newpage

\begin{prop}
\label{prop_basic}
Let $S$ be a monoid.
\begin{enumerate}
\item[\emph{(i)}]
Let $s,t \in S$ with $s \gr t$, and let $p,q \in S$ such that $sp = t$ and $tq = s$.
Then the mapping $x \mapsto xp$ is a bijection from the $\gh$-class of $s$ to the $\gh$-class of $t$, and its inverse is the mapping $x \mapsto xq$.
\item[\emph{(ii)}]
If $s, p_1, p_2 \in S$ and $s p_1 p_2 \gr s$ then $sp_1 \gr s$.
\item[\emph{(iii)}]
The relation $\gr$ is a left congruence, i.e. for all $s, t_1, t_2 \in S$, $t_1 \gr t_2$ implies $st_1 \gr st_2$.
\item[\emph{(iv)}]
For all $s \in S$ the set $sS$ is a union of $\gr$-classes.
\end{enumerate}
The left-right dual statements hold for $\gl$-classes.
\end{prop}

Recall from Section~\ref{sec_main_result} that associated with every $\gh$-class $H$ of a monoid $S$ is a group $\gG(H)$, called the (right) Sch\"{u}tzenberger group of $H$. We now outline some of the basic properties of Sch\"{u}tzenberger groups that we shall need in what follows. For more details, and in particular for proofs of these basic facts, we refer the reader to \cite[Section~2.3]{Lallement1}.

\begin{prop}
Let $S$ be a monoid, let $H$ an be an $\gh$-class of $S$, and let $h_0 \in H$ be an arbitrary element. Then:
\begin{enumerate}
\item[\emph{(i)}]
$\mathrm{Stab}(H) = \{ s \in S: h_0 s \gh h_0 \}$;
\item[\emph{(ii)}]
$\sigma(H) = \{ (s,t) \in \mathrm{Stab}(H) \times \mathrm{Stab}(H) : h_0 s = h_0 t  \}$;
\item[\emph{(iii)}]
$H = h_0 \mathrm{Stab}(H)$.
\end{enumerate}
\end{prop}

\section{A presentation for the Sch\"{u}tzenberger group}
\label{sec_presentation}

\sloppypar{
Modifying the approach used in \cite{Ruskuc2000}, in this section we give a presentation for an arbitrary Sch\"{u}tzenberger group of a monoid given by a presentation. The presentation we shall ultimately obtain (in Theorem~\ref{thm_thepresentation}) is very closely modelled on the presentation given in \cite[Theorem~4.2]{Ruskuc2000}. However, the presentation we give here is better, in the sense that the set of relations we take here is a subset of the set of relations given in the presentation in \cite[Theorem~4.2]{Ruskuc2000},
and there are circumstances where the presentation given here is finite, while the corresponding presentation obtained from \cite[Theorem~4.2]{Ruskuc2000} is infinite. Specifically, as explained in the comments immediately following the statement of Theorem~\ref{thm_main_result}, the presentation given in Theorem~\ref{thm_thepresentation} is finite when applied to the case of a maximal subgroup whose $\gr$-class has only finitely many $\gh$-classes. Consequently Theorem~\ref{thm_thepresentation} encompasses both of the results \cite[Corollary~4.3]{Ruskuc2000} and \cite[Corollary~3]{Ruskuc1999}. 
}

Once obtained, we shall go on to use our presentation below (in Section~\ref{sec_hbase}) to investigate homotopy bases and $\mathrm{FDT}$ for Sch\"{u}tzenberger groups---the main subject of this paper. 

Our approach is to use rewriting methods. An account of such methods in classical combinatorial group theory may be found in \cite[Section 2.3]{Magnus1}. Given a generating set, or presentation, or homotopy base for the monoid, our aim will always be to rewrite the given set into a corresponding set for the Sch\"{u}tzenberger group. Of course the hierarchy must be respected, we need the generating set before the presentation, and need the presentation before the homotopy base.

Let us now fix some definitions and notation that will remain in force throughout the rest of the article.
The exposition in this section follows closely that of \cite{Ruskuc2000} and, to the greatest extent possible, we adopt the same notation and conventions as in \cite{Ruskuc2000}.  

Let $S$ be a monoid given by a presentation $\gp=\lb A | \mathfrak{R} \rb$, and let $H$ be a fixed arbitrary $\gh$-class of $S$, and let $R$ be the $\gr$-class of $S$ that contains $H$. Set
\begin{align*}
 & T  =   \mathrm{Stab}(H)  =  \{  s \in S: Hs = H \}, \\
 & \sigma(H)  =  \{ (s,t) \in \mathrm{Stab}(H) \times \mathrm{Stab}(H) : hs = ht \ (\forall h \in H)  \},
\end{align*}
and let $\gG = \gG(H) = T / \sigma$ denote the (right) Sch\"{u}tzenberger group of $H$.  Let $h \in A^*$ be a fixed word representing some element of $H$.

Let $\{ H_\lambda : \lambda \in \Lambda \}$ be the set of all $\gh$-classes contained in $R$. For each $\lambda \in \Lambda$ fix words $p_\lambda, p_\lambda' \in A^*$ such that:
\[
H p_\lambda = H_{\lambda}, \quad
h_1 p_\lambda p_\lambda' = h_1, \quad
h_2 p_\lambda' p_\lambda = h_2 \quad
(\lambda \in \Lambda, \; h_1 \in H, \; h_2 \in H_\lambda).
\]
This is possible by Proposition~\ref{prop_basic}. Without loss of generality we assume that $\Lambda$ contains a distinguished element $1$, and that $H_1 = H$ and $p_1 \equiv p_1' \equiv 1$, where $1$ denotes the empty word.

By Proposition~\ref{prop_basic}, for every $\lambda \in \Lambda$ and every $s \in S$ either $H_{\lambda}s = H_\mu$, for some $\mu \in \Lambda$, or otherwise $H_{\lambda}s \cap R = \varnothing$, in which case  $H_{\lambda}st \cap R = \varnothing$ for all $t \in S$. Therefore the monoid $S$ acts on the set $\Lambda \cup \{ 0 \}$ via
\[
\lambda \cdot s =
\begin{cases}
\mu & \mbox{if $H_{\lambda} s = H_{\mu}$ with $\lambda, \mu \in \Lambda$,} \\
0 & \mbox{otherwise}.
\end{cases}
\]
Directly generalising the classical method for obtaining Schreier generators for a subgroup of a group, it is shown in \cite[Proposition~4.1]{Ruskuc1999} that the set
\[
\{ p_\lambda a p_{\lambda \cdot a}' / \sigma : \lambda \in \Lambda, \; a \in A, \; \lambda \cdot a \neq 0  \}
\]
generates the Sch\"{u}tzenberger group $\gG(H)$.  
It is with respect to the above generating set that we shall write a presentation for the  Sch\"{u}tzenberger group $\gG(H)$. With this in mind, we define an abstract set of letters:
\[
B = \{ b[\lambda,a] : \lambda \in \Lambda, \; a \in A, \; \lambda \cdot a \neq 0 \}
\]
and a homomorphism $\psi$ from the free monoid $B^*$ into $A^*$ by
\[
\psi: B^* \rightarrow A^*, \quad b[\lambda,a] \mapsto p_\lambda a p'_{\lambda \cdot a}.
\]
The mapping $\psi$ is called the \emph{representation mapping}.
Also, define a mapping:
\[
\phi: \{ (\lambda,w) \in \Lambda \times A^* \; : \; \lambda \cdot w \neq 0 \}   \
\rightarrow
B^*,
\]
inductively by
\begin{equation}
\phi(\lambda, 1) = 1, \quad
\phi(\lambda, aw) = b[\lambda,a] \phi(\lambda \cdot a, w).
\end{equation}
The mapping $\phi$ is called the \emph{rewriting mapping}. The definition of $\phi$ is motivated by the Schreier rewriting process given in the proof of \cite[Proposition~4.1]{Ruskuc1999}.

Given $\lambda \in \Lambda$ and $a \in A$ such that $\lambda \cdot a \neq 0$ we have that the word
$h p_\lambda a p_{\lambda \cdot a}'$ represents and element of $H$ and so we may choose and fix a word $\pi(b[\lambda,a]) \in A^*$ such that
\[
h p_\lambda a p_{\lambda \cdot a}'
=
\pi(b[\lambda,a])h
\]
in $S$. Then we extend the mapping $b[\lambda,a] \mapsto \pi(b[\lambda,a])$ to a homomorphism $\pi: B^* \rightarrow A^*$.

Recall that the relation $\gr$ is a left congruence on $S$. As explained in Section~\ref{sec_intro}, it follows that there is a natural left action $(s,R') \mapsto s * R'$ of $S$ on the set $S / \gr$ of all $\gr$-classes given by:
\[
s * R' = R'' \Leftrightarrow sR' \subseteq R'' \ (s \in S, R', R'' \in S / \gr).
\]
Let $\{ R_i : i \in I \}$ be the inverse orbit of $R$ under this action, i.e. let it be the set:
\[
\{ R' \in S / \gr : (\exists s \in S)(s * R' = R)  \}.
\]
The action of $S$ on $S / \gr$ induces a partial action on $\{ R_i : i \in I \}$ which in turn translates into an action of $S$ on the set $I \cup \{ 0 \}$ given by:
\[
s * i
=
\begin{cases}
j & \mbox{if $i,j \in I$ and $s * R_i = R_j$}, \\
0 & \mbox{otherwise}.
\end{cases}
\]
For each $i \in I$ choose a word $r_i \in A^*$ representing an element of the $\gr$-class $R_i$. Suppose, without loss of generality, that there are two distinguished elements $1, \omega \in I$ such that:
\[
1_S \in R_1, \quad
r_1 \equiv 1, \quad
R = R_{\omega}, \quad
r_{\omega} \equiv h.
\]
For every $a \in A$ and $i \in I$, if $a * i \neq 0$ then $ar_i \in a * R_i = R_{a * i}$. Therefore we can choose words $\kappa(a,i) \in A^*$ such that the relations:
\[
ar_i = r_{a * i} \kappa(a,i) \quad
(a \in A, \; i \in I)
\]
hold in $S$. We extend this to a mapping:
\[
\kappa :
\{ (w,i) \in A^* \times I : w * i \neq 0 \}
\rightarrow
A^*
\]
inductively by
\[
\kappa(1,i) = 1, \quad
\kappa(wa,i) = \kappa(w,a*i) \kappa(a,i).
\]
Up to this point 
in this section
we have followed exactly the development of \cite[Section~4]{Ruskuc1999}. The next step is where we introduce a change.

Next, let $e \in A^*$ be a word representing an element from $\mathrm{Stab}(H)$ such that $e / \sigma = 1$ in $\gG(H)$, in other words $e$ is a word representing an element in the pointwise stabilizer of $H$ under the action of $S$ on itself via right multiplication.
Note that in general $e$ \emph{need not be an idempotent}.
Such an $e$ clearly exists, since one may always take $e$ to be the empty word.
Indeed, if one were to choose $e\equiv 1$ then the rest of the exposition below becomes identical to that of \cite[Section~4]{Ruskuc1999}. Without loss of generality, suppose that $I$ contains a distinguished element $\eta$ such that
\[
e \in R_{\eta}, \quad
r_\eta \equiv e.
\]
Now let $\{ R_k : k \in K \}$ be the orbit of $R_\eta$ under the left action of $S$, i.e. let it be the set $\{ s * R_\eta : s \in S \} $. Then set $J = I \cap K$ noting that $\omega \in J$ since $he= h$. Clearly the action of $S$ on $I \cup \{ 0 \}$ restricted to the set $J \cup \{ 0 \}$ gives rise to a well-defined action of $S$ on the set $J \cup \{ 0 \}$ (in other words, $J \cup \{ 0 \}$ is a subact of $I \cup \{ 0 \}$).

We are now in a position to give a presentation for the Sch\"{u}tzenberger group $\gG$.
\begin{thm}
\label{thm_thepresentation}
With the above notation the Sch\"{u}tzenberger group $\gG$ of $H$ is defined by the presentation with generators $B$ and relations:\\
\begin{tabular}{lc}
$\begin{array}{l}
(\mathrm{R1}) \\
(\mathrm{R2}) \\
(\mathrm{R3}) \\
(\mathrm{R4})
\end{array}$ &
$\begin{array}{ll}
 \phi(\lambda, u) = \phi(\lambda,v)
& (\lambda \in \Lambda, \; (u=v) \in \mathfrak{R}, \; \lambda \cdot u \neq 0), \\
 \phi(\lambda, \kappa(u,j)) = \phi(\lambda, \kappa(v,j))
&  (\lambda \in \Lambda, \; j \in J, \; (u=v) \in \mathfrak{R}, \; H_{\lambda} \subseteq Sr_{u * j}), \\
  b[\lambda,a] = \phi(1, \kappa( \pi( b[\lambda,a]), \omega))
&  (\lambda \in \Lambda, \; a \in A, \; \lambda \cdot a \neq 0),  \\
  1 = \phi(1, \kappa(h,\eta)).
\end{array}$
\end{tabular}
\end{thm}
The above presentation is finite provided $\Lambda$, $J$, $A$, and $\mathfrak{R}$ are all finite. Therefore Theorem~\ref{thm_main_result}(i) is an immediate corollary of Theorem~\ref{thm_thepresentation}.

\subsection*{Proof of Theorem~\ref{thm_thepresentation}}
The rest of this section will be devoted to proving Theorem~\ref{thm_thepresentation}.
The result is proved by refining slightly the arguments given in \cite{Ruskuc2000}, and in many cases lemmas from that paper can (and will) simply be quoted. One must also perform a slight notation translation, since we use $\kappa$ here where $\tau$ is used in \cite{Ruskuc2000}.

The proof has two parts: first we show that all of the relations (R1)--(R4) hold in $\mathcal{G}$, and then we show that they suffice to define $\mathcal{G}$. 

Recall that a generator $b[\lambda,a]$ from $B$ represents the element $({p_\lambda a p_{\lambda \cdot a}'}) / \sigma = {\psi(b[\lambda,a])} / \sigma$ of the group $\gG$. So given $\alpha, \beta \in B^*$ the relation $\alpha = \beta$ holds in $\gG$ if and only if ${\psi(\alpha)} / \sigma = {\psi(\beta)} / \sigma$ which by the definition of Sch\"{u}tzenberger group is equivalent to $h \psi(\alpha) = h \psi(\beta)$ holding in $S$. 

Since many of the definitions given above come from \cite{Ruskuc2000} we may now quote some lemmas from \cite{Ruskuc2000}, without modification, that will prove useful. 
The first two such lemmas tell us that the maps $\phi$ and $\kappa$ are both very nicely behaved, in the sense  that each of them behaves like a homomorphism.
\begin{lem}[{\cite[Lemma~5.1]{Ruskuc2000}}]\label{Ruskuc_Lemma5.1} \noindent

\begin{enumerate}
\item[\emph{(i)}]
For every $w_1, w_2 \in A^*$ and every $\lambda \in \Lambda$ such that $\lambda \cdot w_1 w_2 \neq 0$ we have
\[
\phi(\lambda, w_1w_2) \equiv
\phi(\lambda, w_1) \phi(\lambda \cdot w_1, w_2).
\]
\item[\emph{(ii)}]
For every $w \in A^*$ and every $\lambda \in \Lambda$ such that $\lambda \cdot w \neq 0$ the relation
\[
h \psi \phi(\lambda, w) = h p_{\lambda} w p_{\lambda \cdot w}'
\]
holds in $S$.
\item[\emph{(iii)}]
If $w_1, w_2 \in A^*$ are such that the relation $w_1 = w_2$ holds in $S$, and if $\lambda \in \Lambda$ is such that $\lambda \cdot w_1 \neq 0$, then the relation $\phi(\lambda, w_1) = \phi(\lambda, w_2)$ is a consequence of the relations \rm{(R1)}-\rm{(R4)}.
\end{enumerate}
\end{lem}

\begin{lem}[{\cite[Lemma~5.2]{Ruskuc2000}}]\label{Ruskuc_Lemma5.2} \noindent \begin{enumerate}
\item[\emph{(i)}]
For every $w_1, w_2 \in A^*$ and $j \in J$ such that $w_1 w_2 * j \neq 0$ we have
\[
\kappa(w_1 w_2, j)
\equiv
\kappa(w_1, w_2 * j) \kappa(w_2, j).
\]
\item[\emph{(ii)}]
For every $w \in A^*$ and $j \in J$ such that $w * j \neq 0$ the relation
\[
w r_j = r_{w * j} \kappa(w,j)
\]
holds in $S$.
\end{enumerate}
\end{lem}

The next lemma shows how $\psi$ and $\pi$ are related.

\begin{lem}[{\cite[Lemma~5.3]{Ruskuc2000}}]
\label{Ruskuc_Lemma5.3}
For every $w \in B^*$ the relation
\[
h \psi(w) = \pi(w) h
\]
holds in $S$.
\end{lem}

The word $\pi(\omega)$ has certain special properties, as the following result shows.

\begin{lem}[{\cite[Lemma~5.4]{Ruskuc2000}}]
\label{Ruskuc_Lemma5.4}
For every $w \in B^*$ we have
\begin{enumerate}
\item[\emph{(i)}]
$\pi(w) * \omega = \omega$; and
\item[\emph{(ii)}]
$1 \cdot \kappa(\pi(w), \omega) = 1$.
\end{enumerate}
\end{lem}

Lemmas~5.5, 5.6 and 5.7 from \cite{Ruskuc2000} all hold verbatim, and they will not be restated here. They simply state that the relations (22)-(24), which in particular include our relations (R1)-(R3),  all hold in the Sch\"{u}tzenberger group $\gG$.

Each of \cite[Lemmas~5.8--5.11]{Ruskuc2000} require some modification for our purposes, which we now explain.
\begin{lem}[Generalising {\cite[Lemma~5.8]{Ruskuc2000}}]
\label{lem_5.8'}
The relation:
\[
\phi(1, \kappa(h, \eta)) = 1
\]
holds in $\gG$.
\end{lem}
\begin{proof}
By definition, $he = h$ in $S$. Now applying Lemma~\ref{Ruskuc_Lemma5.2} we obtain:
\[
h = he \equiv hr_\eta
=
r_{h * \eta} \kappa(h, \eta)
=
h \kappa(h, \eta),
\]
and hence $1 \cdot \kappa(h,\eta) = 1$. Then combining the above equation with Lemma~\ref{Ruskuc_Lemma5.1}(ii) gives
\[
h \psi( \phi(1, \kappa(h,\eta)) ) =
h p_1 \kappa(h, \eta) p_1' \equiv h \kappa(h, \eta) = h \equiv h \psi(1),
\]
and thus $\phi(1, \kappa(h, \eta)) = 1$ holds in $\gG$, as required.
\end{proof}

This completes the proof that all the relations (R1)-(R4) hold in $\gG$. Now we shall work towards the proof that this set of relations is enough to actually define $\gG$.

\begin{lem}[Generalising {\cite[Lemma~5.9]{Ruskuc2000}}]
\label{lem_5.9'}
For all $w \in B^*$ the relation
\[
\phi(1, \kappa(\pi(w) h, \eta)) = w
\]
is a consequence of the relations \emph{(R1)-(R4)}.
\end{lem}
\begin{proof}
The result is proved by induction on the length of the word $w$.
When $|w| = 0$ then the displayed equation is just relation (R4).
Next consider the case $|w|=1$, so $w \equiv b[\lambda,a]$ for some $\lambda \in \Lambda$ and $a \in A$ with $\lambda \cdot a \neq 0$.
Then applying Lemma~\ref{Ruskuc_Lemma5.2}(i), and observing $h * \eta = \omega$ (since $he=h$),  we obtain
\[
\kappa( \pi(b[\lambda,1])h, \eta)
\equiv
\kappa(\pi(b[\lambda,a]), h*\eta)
\kappa(h, \eta)
\equiv
\kappa(\pi(b[\lambda,a]), \omega)
\kappa(h, \eta).
\]
By Lemma~\ref{Ruskuc_Lemma5.4}(ii) we have:
\[
1 \cdot \kappa(\pi(b[\lambda,a]), \omega) = 1,
\]
and thus applying Lemma~\ref{Ruskuc_Lemma5.1}(i) we have
\begin{align*}
      			& \;\phantom{\equiv}\; \phi(1, \kappa(\pi(b[\lambda,a]), \omega) \kappa(h, \eta)) & \\
 			& \equiv\phi(1, \kappa(\pi(b[\lambda,a]), \omega) \phi(1, \kappa(h, \eta)) & \\
	 		& =\phi(1, \kappa(\pi(b[\lambda,a]), \omega) &  \mbox{(by (R4))} \\
			& =b[\lambda,a], & \mbox{(by (R3))}
\end{align*}
as required. For the induction step, assume $|w| > 1$ and write $w \equiv w_1 w_2$ with $|w_1|, |w_2| > 0$. Then, recalling that $\pi$ is a homomorphism and that $h * \eta = \omega$, we have:
\begin{align*}
      			& \;\phantom{\equiv}\; \phi(1, \kappa(\pi(w) h, \eta))
			& & \\
			& \equiv \phi(1, \kappa(\pi(w_1), \omega) \kappa(\pi(w_2), \omega) \kappa(h,\eta))
			& & \mbox{(by Lemmas~\ref{Ruskuc_Lemma5.2}(i) and \ref{Ruskuc_Lemma5.4})} \\
			& \equiv \phi(1, \kappa(\pi(w_1), \omega)) \phi(1, \kappa(\pi(w_2), \omega)) \phi(1, \kappa(h,\eta))
			& & \mbox{(by Lemmas~\ref{Ruskuc_Lemma5.1}(i) and \ref{Ruskuc_Lemma5.4})} \\			
			& = \phi(1, \kappa(\pi(w_1), \omega)) \phi(1, \kappa(h,\eta)) \phi(1, \kappa(\pi(w_2), \omega)) \phi(1, \kappa(h,\eta))
			& & \mbox{(by (R4)} \\
			& \equiv \phi(1, \kappa(\pi(w_1) h, \eta)) \phi(1, \kappa(\pi(w_2) h, \eta))
			& & \mbox{(by Lemmas \ref{Ruskuc_Lemma5.1}(i), \ref{Ruskuc_Lemma5.2}(i) and \ref{Ruskuc_Lemma5.4})} \\			
			& = w_1w_2 \equiv w,
			& & \mbox{(by induction)}
\end{align*}
as required.
\end{proof}

The following two lemmas, and their proofs, will be of fundamental importance when we turn our attention to homotopy in the next section.
\begin{lem}[Generalising {\cite[Lemma~5.10]{Ruskuc2000}}]
\label{lem_5.10'}
Let $\alpha, \beta \in A^*$ and $(u=v) \in \mathfrak{R}$ such that $\alpha u \beta$ represents an element of the $\gh$-class $H$.
Then:
\begin{enumerate}
\item[\emph{(i)}]
$\alpha u \beta * \eta = \omega \neq 0$;
\item[\emph{(ii)}]
$1 \cdot \kappa(\alpha u \beta, \eta) = 1$; and
\item[\emph{(iii)}]
the relation $\phi(1, \kappa(\alpha u \beta, \eta)) = \phi(1, \kappa(\alpha v \beta, \eta))$ is a consequence of the relations \emph{(R1)-(R4)}.
\end{enumerate}
\end{lem}
\begin{proof}
(i) Since $\alpha u \beta \in H$, from the definition of $e$ we have $\alpha u \beta e = \alpha u \beta$ in $S$. It follows that $\alpha u \beta * \eta = \omega$. (ii) Applying the previous statement (i) and Lemma~\ref{Ruskuc_Lemma5.2}(ii) we obtain
\[
H \ni \alpha u \beta = \alpha u \beta e \equiv \alpha u \beta r_\eta = r_{\omega} \kappa(\alpha u \beta, \eta) \equiv h \kappa(\alpha u \beta, \eta).
\]
It follows that:
\[
1 \cdot \kappa(\alpha u \beta, \eta) = 1.
\]
(iii) The argument follows very similar lines to the proof of \cite[Lemma~5.10]{Ruskuc2000}. Applying Lemmas~\ref{Ruskuc_Lemma5.1}(i) and \ref{Ruskuc_Lemma5.2}(i) to the left hand side gives:
\begin{align*}
			\phi(1, \kappa(\alpha u \beta, \eta)) 
			 \equiv [\phi(1, \kappa(\alpha, u \beta * \eta))] \;
			[\phi(1 \cdot \kappa(\alpha, u\beta * \eta), \kappa(u, \beta * \eta))] \;
			[\phi(1 \cdot \kappa(\alpha u, \beta * \eta), \kappa(\beta, \eta))],
\end{align*}
and similarly from the right hand side we obtain:
\begin{align*}
			\phi(1, \kappa(\alpha v \beta, \eta)) 
			 \equiv [\phi(1, \kappa(\alpha, v \beta * \eta))] \;
			[\phi(1 \cdot \kappa(\alpha, v\beta * \eta), \kappa(v, \beta * \eta))] \;
			[\phi(1 \cdot \kappa(\alpha v, \beta * \eta), \kappa(\beta, \eta))].
\end{align*}
Let us consider each of the three pairs of terms in turn.
Since $u=v$ in $S$ it follows that $u \beta * \eta = v \beta * \eta$, and hence
\[
\phi(1, \kappa(\alpha, u \beta * \eta)) \equiv \phi(1, \kappa(\alpha, v \beta * \eta)).
\]
This deals with the first pair of terms.

For the second (middle) pair of terms, we claim that
\begin{equation}\label{eq_newstar}
\phi(1 \cdot \kappa(\alpha, u\beta * \eta), \kappa(u, \beta * \eta))
=
\phi(1 \cdot \kappa(\alpha, v\beta * \eta), \kappa(v, \beta * \eta))
\end{equation}
is one of the relations (R2). This relies on the assumption that $\alpha u \beta = \alpha v \beta$ represents an element of $H$.
Indeed, by part (i), $\alpha u \beta * \eta = \omega$, which implies $u \beta * \eta \in J$ and $\beta * \eta \in J$. Then since:
\[
\alpha * (u \beta * \eta) = \alpha u \beta * \eta = \omega \neq 0,
\]
by part (i) again, we may apply Lemma~\ref{Ruskuc_Lemma5.2}(ii) to deduce:
\[
\alpha r_{u \beta * \eta} = r_{\alpha u \beta * \eta} \kappa(\alpha, u\beta * \eta) = r_{\omega} \kappa(\alpha, u \beta * \eta) = h \kappa(\alpha, u \beta * \eta).
\]
Now by definition since $e \equiv r_\eta \in R_\eta$ we have
\[
r_{u \beta * \eta}   \; \gr \; u \beta e ,
\]
which, since $\gr$ is a left congruence, implies
\[
\alpha r_{u \beta * \eta} \; \gr \;  \alpha u \beta e =  \alpha u \beta  \in H.
\]
It follows that $h \kappa(\alpha, u \beta * \eta) = \alpha r_{u \beta * \eta} \in R$ and hence $1 \cdot \kappa(\alpha, u \beta * \eta) \neq 0$.
Then since $u=v$ in $S$ we have that $u\beta * \eta = v\beta * \eta$ and thus:
\[
0 \neq 1 \cdot \kappa(\alpha, u\beta * \eta) = 1 \cdot \kappa(\alpha, v\beta * \eta) \in \Lambda.
\]
We have shown
\[
H_{1 \cdot  \kappa(\alpha, u \beta * \eta)} \ni h \kappa(\alpha, u \beta * \eta) = \alpha r_{u \beta * \eta} \in Sr_{u \beta * \eta}.
\]
Therefore $H_{1 \cdot  \kappa(\alpha, u \beta * \eta)} \cap Sr_{u \beta * \eta} \neq \varnothing$ and so:
\[
H_{1 \cdot  \kappa(\alpha, u \beta * \eta)} \subseteq  Sr_{u \beta * \eta} = Sr_{u * (\beta * \eta)}.
\]
It follows that, as claimed, \eqref{eq_newstar} is one of the relations (R2).

Finally consider the third pair of terms. By part (i) above and Lemma~\ref{Ruskuc_Lemma5.2}(ii), since $u=v$ in $S$ we have:
\[
h \kappa(\alpha u, \beta * \eta)
\equiv
r_{\alpha u \beta*\eta} \kappa(\alpha u, \beta * \eta)
=
\alpha u r_{\beta * \eta}
=
\alpha v r_{\beta * \eta}
\equiv
r_{\alpha v \beta*\eta} \kappa(\alpha u, \beta * \eta)
\equiv
h \kappa(\alpha v, \beta * \eta),
\]
so
\[
1 \cdot \kappa(\alpha u, \beta * \eta) = 1 \cdot \kappa(\alpha v, \beta * \eta),
\]
and hence
\[
\phi(1 \cdot \kappa(\alpha u, \beta * \eta), \kappa(\beta, \eta))
\equiv
\phi(1 \cdot \kappa(\alpha v, \beta * \eta), \kappa(\beta, \eta)).
\]
We conclude from these three observations that $\phi(1, \kappa(\alpha u \beta, \eta)) = \phi(1, \kappa(\alpha v \beta, \eta))$ may be deduced by a single application of a relation of the form (R2), and so this completes the proof of the theorem. \end{proof}
The proof of Theorem~\ref{thm_thepresentation} is then finished off using the following lemma.
\begin{lem}[Generalising {\cite[Lemma~5.11]{Ruskuc2000}}]
\label{lem_5.11'}
If $w_1, w_2 \in B^*$ are any two words such that $w_1 = w_2$ holds in $\gG$ then the relation $w_1 = w_2$ is a consequence of the relations \emph{(R1)-(R4)}.
\end{lem}
\begin{proof}
Since $w_1 = w_2$ it follows from Lemma~\ref{Ruskuc_Lemma5.3} that $\pi(w_1) h = \pi(w_2) h$ in $S$. Therefore, there is a sequence of words from $A^*$
\[
\pi(w_1) h \equiv  \gamma_1, \gamma_2, \ldots, \gamma_n \equiv  \pi(w_2) h
\]
such that each $\gamma_{i+1}$ is obtained from $\gamma_i$ by a single application of a relation from $\mathfrak{R}$. Since $\pi(w_1) h$ represents an element of $H$ (by Lemma~\ref{Ruskuc_Lemma5.4}(i)) it follows that every word in this sequence represents this same fixed element of $H$. Therefore we may apply Lemma~\ref{lem_5.10'} which tells us  that each relation:
\[
\phi(1, \kappa(\gamma_i, \eta)) =
\phi(1, \kappa(\gamma_{i+1}, \eta))
\]
is a consequence of the relations (R1)-(R4). Therefore the relation:
\[
\phi(1, \kappa(\pi(w_1) h, \eta)) = \phi(1, \kappa(\pi(w_2) h, \eta))
\]
is a consequence of the relations (R1)-(R4). By Lemma~\ref{lem_5.9'} the relations:
\[
w_1 = \phi(1, \kappa(\pi(w_1) h, \eta))
\quad \& \quad
\phi(1, \kappa(\pi(w_2) h, \eta)) = w_2,
\]
are each also consequences of the relations (R1)-(R4). Combining these observations we conclude that $w_1 = w_2$ is a consequence of the relations (R1)-(R4), completing the proof of the lemma.
\end{proof}

This completes the proof of Theorem~\ref{thm_thepresentation}, and concludes this section.

\section{A homotopy base for the Sch\"{u}tzenberger group}
\label{sec_hbase}

In the previous section we saw how to rewrite a presentation for a monoid to obtain a presentation for 
an arbitrary Sch\"{u}tzenbegrer group of that monoid, which is given in the statement of Theorem~\ref{thm_thepresentation}. The aim of this section is to find a homotopy base for the Sch\"{u}tzenbegrer group with respect to the presentation of Theorem~\ref{thm_thepresentation}. Then as a corollary we shall deduce Theorem~\ref{thm_main_result}.

We continue to use the same notation as in the previous section. In particular, $\gp = \lb A | \mathfrak{R} \rb$ is a presentation defining the monoid $S$, $\gq = \lb B | \mathfrak{U} \rb$ is the presentation for the Sch\"{u}tzenberger group $\gG = \gG(H)$ given in Theorem~\ref{thm_thepresentation}, and we have the index sets $J$ (indexing a specific collection of $\gr$-classes, defined above) and $\Lambda$ (indexing the $\gh$-classes in the $\gr$-class containing $H$), and the mappings $\phi$, $\psi$, $\kappa$ and $\pi$.

In the previous section we defined several mappings which, in various ways, rewrote words written in the alphabet $A$ into words written in the alphabet $B$, and vice versa. Since $A^*$ and $B^*$ are the sets of vertices of the derivations graphs $\Gamma(\gp
)$ and $\Gamma(\gq)$, respectively, these maps are actually mappings between the vertices of $\Gamma(\gp)$ and those of $\Gamma(\gq)$.
To study homotopy we need to move from working with vertices to working with paths. So, our first task here will be to extend some of the definitions from the previous section to obtain mappings between paths of $\Gamma(\gp)$ and paths of $\Gamma(\gq)$. 

Let us now outline the main steps of the proof that is to follow. 

\begin{enumerate}[(I)]
\item
We being by extending the mapping $\phi$ in such a way that it maps edges from $\Gamma(\mathcal{P})$ to edges in $\Gamma(\mathcal{Q})$, and then extend this mapping in a natural way so that is maps paths to paths. In fact, in its most general form, the function $\phi$ will be defined on triples $(\lambda, \bbe, j)$ where $\lambda \in \Lambda$, $j \in J$, $\bbe$ is an edge from $\Gamma(\gp)$ and conditions (i)--(iii) of \eqref{eqn_necessary} (given below) hold. 
\item 
In the second step we define a mapping $\theta$ which maps each edge from $\Gamma(\gq)$ to a path in $\Gamma(\mathcal{P})$.  
\item 
Next, for each $w \in B^*$ we shall define a certain path $\Lambda_w$ in 
$\Gamma(\gq)$. 
\item 
The definitions of $\phi$, $\theta$ and $\Lambda$ given in steps (I)--(III) will be made in such a way that for every edge $\bbe$ of $\Gamma(\gq)$:
\[
\bbe \circ \; \Lambda_{\tau \bbe} \circ \phi (\theta (\bbe))^{-1} \circ \Lambda_{\iota \bbe}^{-1}
\] 
will be a well-defined closed path in the graph $\Gamma(\gq)$. We use $Z$ to denote the set of all such paths, as $\bbe$ ranges over all edges of $\Gamma(\gq)$.
\item Given a homotopy base $X$ for $S$, which is a set of closed paths from $\Gamma(\gp)$, applying the mapping $\phi$ we obtain a set $Y_1$ of closed paths in $\Gamma(\gq)$. It is then straightforward to verify (see Lemma~\ref{lem_infinitetriv}) that taking $Z$ and $Y_1$ together we obtain a (necessarily infinite) homotopy base for $\Gamma(\gq)$.  
\item
In the final part of the proof we show how one may replace the set $Z$ by a certain subset $Y_2 \cup Y_3$, with the property that when the presentation $\lb A | \mathfrak{R} \rb$ is finite, and $X$ is finite, the set $Y_1 \cup Y_2 \cup Y_3$ is a finite homotopy base for $\lb B | \mathfrak{U} \rb$ (see Theorem~\ref{thm_baseforG}), thus establishing our main result Theorem~\ref{thm_main_result}(ii).  
\end{enumerate}

Before starting with step (I), we shall make a general observation that will be useful for us.

\begin{lem}
\label{lem_generalfact}
Let $u_1, u_2 \in A^*$ such that $u_1 = u_2$ in $S$, and let $j \in J$ with $u_1 * j \neq 0$ (equivalently $u_2 * j \neq 0$).
For all $\lambda \in \Lambda$, if
$
S r_{u_1 * j} \supseteq H_\lambda
$
then
\[
\lambda \cdot \kappa(u_1,j) =
\lambda \cdot \kappa(u_2,j) \neq 0.
\]
\end{lem}
\begin{proof}
Since $u_1 * j \neq 0$, by Lemma~\ref{Ruskuc_Lemma5.2}(ii), in $S$ we have
\[
u_1 r_j = r_{u_1 * j} \kappa(u_1, j)
\quad
\&
\quad
u_2 r_j = r_{u_2 * j} \kappa(u_2, j).
\]
Since $hp_\lambda$ represents an element from $H_\lambda$, it follows from the assumptions of the lemma that there exists $q \in A^*$ such that
$
q r_{u_1 * j} \equiv q r_{u_2 * j} = h p_{\lambda}
$
in $S$. Hence since $u_1 = u_2$ in $S$, we have:
\[
(hp_\lambda) \kappa(u_1,j) =
q r_{u_1 * j} \kappa(u_1,j) =
qu_1 r_j =
qu_2 r_j =
q r_{u_2 * j} \kappa(u_2,j) =
(h p_\lambda) \kappa(u_2, j)
\]
in $S$, and thus
\[
\lambda \cdot \kappa(u_1, j) = \lambda \cdot \kappa(u_2,j).
\]
As $u_1 r_j \gr r_{u_1 * j}$ and $\gr$ is a left congruence it follows that $q u_1 r_j \gr q r_{u_1 * j}$ which gives
\[
hp_\lambda \kappa(u_1, j) =
q u_1 r_j \gr q r_{u_1 * j}
= hp_\lambda \in H_\lambda,
\]
and therefore $\lambda \cdot \kappa(u_1, j) \neq 0$.
\end{proof}

Let $X$ be a fixed homotopy base for $\gp$. Our aim is to write down a homotopy base for the Sch\"{u}tzenberger group $\gG$. This will ultimately be done in Theorem~\ref{thm_baseforG} below, which has Theorem~\ref{thm_main_result} as an immediate corollary. We prove the result by working through steps (I)-(VI) outlined above. 

\subsection*{\boldmath (I) The function $\phi(\lambda, \kappa(\bbe,j))$ (where $\bbe \in \Gamma(\gp)$)}

\noindent The function $\phi$ will map edges from $\Gamma(\mathcal{P})$ to edges in $\Gamma(\mathcal{Q})$, and will then be extended in the obvious way to a mapping on paths. The definition of $\phi$ arises naturally from the proof of Lemma~\ref{lem_5.10'}. For every edge $\bbf$ of $\Gamma(\gp)$, $\lambda \in \Lambda$ and $j \in J$ such that
\begin{equation}
\label{eqn_necessary}
\mbox{(i)} \; \iota \bbf * j \neq 0,
\quad \quad
\mbox{(ii)} \; \lambda \cdot \kappa(\iota \bbf, j) \neq 0
\quad \quad \&
\quad \quad
\mbox{(iii)} \; H_\lambda \subseteq S r_{\iota \bbf * j},
\end{equation}
writing $\bbf = \alpha \cdot \bba \cdot \beta$ where $\alpha, \beta \in A^*$ and $\bba = \bba_r^\epsilon$ is elementary, we define 
\begin{equation}
\label{eqn_phipath0}
\phi(\lambda, \kappa(\alpha \cdot \bba \cdot \beta, j) )
=
\bba_{
\phi(\lambda, \kappa(\iota \bba, j)) =
\phi(\lambda, \kappa(\tau \bba, j))
}^\epsilon
\end{equation}
if $\alpha \equiv \beta \equiv 1$, and
\begin{align}
\label{eqn_phipath}
&  \phi(\lambda, \kappa(\alpha \cdot \bba \cdot \beta, j) )  \\
= &  \phi(\lambda, \kappa(\alpha, \iota \bba \beta * j ))  \cdot 
	\phi(\lambda \cdot \kappa(\alpha, \iota \bba \beta * j), \kappa(\bba, \beta * j))  \cdot 
		\phi(\lambda \cdot \kappa(\alpha \iota \bba, \beta * j), \kappa(\beta, j)) 
		\label{eqn_phipath2}
\end{align}
otherwise.

In this definition, the right hand term in \eqref{eqn_phipath0} is an elementary edge from $\Gamma(\gq)$ corresponding to a relation of the form (R2) (this will be verified in Lemma~\ref{lem_propertiesofphi} below), and the middle term in \eqref{eqn_phipath2} is defined since it is of the form \eqref{eqn_phipath0}. Note that $\phi$ maps edges of $\Gamma(\gp)$ to edges of $\Gamma(\gq)$.

\begin{remark}
\label{remark_implies}
Note that in \eqref{eqn_necessary} (and similarly in \eqref{eqn_necessary2} below), by Lemma~\ref{lem_generalfact}, condition (iii) actually implies condition (ii). 
\end{remark}
\noindent Note the slight abuse of notation here, now the function $\phi$ may be applied to edges and to words, in the latter case the definition comes from Section~\ref{sec_presentation}, while in the former it is given by the definition immediately above. This choice of notation is justified by \eqref{eqn_initialandterminal} below.

\begin{lem}
\label{lem_propertiesofphi}
Let $\bbf$ be an edge of $\Gamma(\gp)$ and let $\lambda \in \Lambda$ and $j \in J$ be such that {\rm(i)}-{\rm(iii)} of \eqref{eqn_necessary} are satisfied.
Then $\phi(\lambda, \kappa(\bbf, i) )$ is an edge in the derivation graph $\Gamma(\gq)$ satisfying:
\begin{equation}
\label{eqn_inverse}
\phi(\lambda, \kappa(\bbf, j))^{-1}
=
\phi(\lambda, \kappa(\bbf^{-1}, j)),
\end{equation}
\begin{equation}
\label{eqn_initialandterminal}
\iota \phi(\lambda, \kappa(\bbf, j) )
\equiv
\phi(\lambda, \kappa(\iota\bbf, j) )
\quad \& \quad
\tau \phi(\lambda, \kappa(\bbf, j) )
\equiv
\phi(\lambda, \kappa(\tau\bbf , j) ).
\end{equation}
\end{lem}
\begin{proof}
Write $\bbf = \alpha \cdot \bba \cdot \beta$, where $\bba$ is elementary and $\alpha, \beta \in A^*$. We must show that $\phi(\lambda, \kappa(\alpha \cdot \bba \cdot \beta, i) )$ is a well-defined edge in the derivation graph $\Gamma(\gq)$, and that it satisfies:
\[
\iota \phi(\lambda, \kappa(\alpha \cdot \bba \cdot \beta, j) )
\equiv
\phi(\lambda, \kappa(\alpha \cdot \iota \bba \cdot \beta, j) )
\;\; \& \;\;
\tau \phi(\lambda, \kappa(\alpha \cdot \bba \cdot \beta, j) )
\equiv
\phi(\lambda, \kappa(\alpha \cdot \tau \bba \cdot \beta, j) ).
\]
First suppose that $\alpha \equiv \beta \equiv 1$. Then  since \eqref{eqn_necessary} holds it follows that
\[
\phi(\lambda, \kappa(\iota \bba, j))
=
\phi(\lambda, \kappa(\tau \bba, j))
\]
is a relation in $\gq$ of the form (R2), and hence \eqref{eqn_phipath0} defines an elementary edge of $\Gamma(\gq)$ satisfying \eqref{eqn_initialandterminal}.

Now suppose that one of $\alpha$ or $\beta$ is not the empty word. We need to check that each of the terms in equation \eqref{eqn_phipath2} is defined, and that \eqref{eqn_initialandterminal} holds.

The word acting on the left in \eqref{eqn_phipath2} is
$\phi(\lambda, \kappa(\alpha, \iota \bba \beta * j ))$ which is defined, since $\iota \bba \beta * j \neq 0$ by assumption \eqref{eqn_necessary}(i), then
$
\lambda \cdot \kappa(\alpha, \iota \bba \beta * j ) \neq 0
$
by Lemma~\ref{Ruskuc_Lemma5.2}(i) and assumption \eqref{eqn_necessary}(ii). On the other hand, the word acting on the right in \eqref{eqn_phipath2} is
$		\phi(\lambda \cdot \kappa(\alpha \iota \bba, \beta * j), \kappa(\beta, j))$ which is defined, since $\beta * j \neq 0$ by \eqref{eqn_necessary}(i), 
and then by Lemma~\ref{Ruskuc_Lemma5.2}(i) and assumption \eqref{eqn_necessary}(ii) we have:
\[
\lambda \cdot \kappa(\alpha \iota \bba, \beta * j) \cdot \kappa(\beta, j)
=
\lambda \cdot \kappa(\alpha \iota \bba \beta, j)
\neq 0.
\]
Finally, the middle term in \eqref{eqn_phipath2} is
\[
\phi(\lambda \cdot \kappa(\alpha, \iota \bba \beta * j), \kappa(\bba, \beta * j))
=
\bba_{
\phi(\lambda', \kappa(\iota \bba, j')) =
\phi(\lambda', \kappa(\tau \bba, j'))
}^\epsilon
\]
where
\[
\lambda' = \lambda \cdot \kappa(\alpha, \iota \bba \beta * j)
\quad \& \quad
j' =  \beta * j.
\]
With these parameters we much check the (i)-(iii) of \eqref{eqn_necessary} all hold.
Clearly (i) holds since $\iota \bba * j' = \iota \bba * \beta * j \neq 0$, so by Remark~\ref{remark_implies} it just remains to show that
$
H_{\lambda'} \subseteq
S r_{\iota \bba * j'}.
$
Let $u \equiv \iota \bba$. By Lemma~\ref{Ruskuc_Lemma5.2}(ii) we have:
\begin{equation}
\label{eqn_diamond}
\alpha r_{u * j'} = r_{\alpha * u * j'} \kappa(\alpha, u * j').
\end{equation}
From assumption \eqref{eqn_necessary}(iii)
\[
H_\lambda \subseteq S r_{\alpha u \beta * j} = S r_{\alpha u * j'}.
\]
So there exists $s \in A^*$ such that $sr_{\alpha u * j'} \in H_\lambda$. Premultiplying \eqref{eqn_diamond} by $s$ gives
\[
s \alpha r_{u * j'} = s r_{\alpha * u * j'} \kappa(\alpha, u * j')
\]
where $sr_{\alpha u * j'} \in H_\lambda$. Since by definition $\lambda \cdot \kappa(\alpha, u*j') = \lambda'$, we deduce that $s \alpha r_{u * j'} \in H_{\lambda'}$, which implies
$H_{\lambda'} \subseteq S r_{\iota \bba * j'}$. This completes the proof that the middle term in the expression \eqref{eqn_phipath2} is an elementary edge of $\Gamma(\gq)$.

To complete the proof of the lemma we must verify \eqref{eqn_inverse} and \eqref{eqn_initialandterminal}.
Straightforward applications of Lemmas~\ref{Ruskuc_Lemma5.1} and \ref{Ruskuc_Lemma5.2} and equations \eqref{eqn_phipath}--\eqref{eqn_phipath2} give
\[
\iota \phi(\lambda, \kappa(\alpha \cdot \bba \cdot \beta))
\equiv
\phi(\lambda, \kappa(\alpha \cdot \iota \bba \cdot \beta)).
\]
For the terminal vertex, we apply assumption \eqref{eqn_necessary}(iii) and Lemma~\ref{lem_generalfact} which together imply:
\[
\lambda \cdot \kappa(\alpha \iota \bba, \beta * j)
=
\lambda \cdot \kappa(\alpha \tau \bba, \beta * j)
\neq 0.
\]
In addition, $\iota \bba = \tau \bba$ implies $\iota \bba \beta * j = \tau \bba \beta * j$. This allows us to replace $\iota$'s by $\tau$'s, concluding
\begin{align*}
& \; \phantom{=} \;
\tau \phi(\lambda, \kappa(\alpha \cdot \bba \cdot \beta, j) )  \\
& \equiv \tau(
\phi(\lambda, \kappa(\alpha, \iota \bba \beta * j )) \; \cdot \;
	\phi(\lambda \cdot \kappa(\alpha, \iota \bba \beta * j), \kappa(\bba, \beta * j)) \; \cdot \;
		\phi(\lambda \cdot \kappa(\alpha \iota \bba, \beta * j), \kappa(\beta, j))
) \\
& \equiv \phi(\lambda, \kappa(\alpha, \iota \bba \beta * j )) \; \cdot \;
	\phi(\lambda \cdot \kappa(\alpha, \iota \bba \beta * j), \kappa(\tau \bba, \beta * j)) \; \cdot \;
		\phi(\lambda \cdot \kappa(\alpha \iota \bba, \beta * j), \kappa(\beta, j)) \\
& \equiv \phi(\lambda, \kappa(\alpha, \tau \bba \beta * j )) \; \cdot \;
	\phi(\lambda \cdot \kappa(\alpha, \tau \bba \beta * j), \kappa(\tau \bba, \beta * j)) \; \cdot \;
		\phi(\lambda \cdot \kappa(\alpha \tau \bba, \beta * j), \kappa(\beta, j)) \\
& \equiv	\phi(\lambda, \kappa(\alpha \cdot \tau \bba \cdot \beta)).
\end{align*}
Finally, it follows easily from the definition that 
\[
\phi(\lambda, \kappa(\bbf, j))^{-1}
=
\phi(\lambda, \kappa(\bbf^{-1}, j)).
\]
This establishes \eqref{eqn_inverse} and \eqref{eqn_initialandterminal}, and so completes the proof.
\end{proof}
The above definition of $\phi$ on edges extends naturally to paths in the following way.
For every path $\bbp = \bbf_1 \circ \bbf_2 \circ \cdots\circ \bbf_m$ of $\Gamma(\gp)$, $\lambda \in \Lambda$ and $j \in J$ such that
\begin{equation}
\label{eqn_necessary2}
\mbox{(i)} \; \iota \bbp * j \neq 0,
\quad \quad
\mbox{(ii)} \; \lambda \cdot \kappa(\iota \bbp, j) \neq 0
\quad \quad \&
\quad \quad
\mbox{(iii)} \; H_\lambda \subseteq S r_{\iota \bbp * j}
\end{equation}
we define
\begin{equation}
\label{eqn_paths}
\phi(\lambda, \kappa(\bbp,  j)) =
\phi(\lambda, \kappa(\bbf_1,j))\circ
\phi(\lambda, \kappa(\bbf_2,j))\circ \cdots\circ
\phi(\lambda, \kappa(\bbf_m,j)).
\end{equation}
Note that since in $S$ we have
$
\iota \bbp \equiv \iota \bbf_1 = \cdots = \iota \bbf_m,
$
applying Lemma~\ref{lem_generalfact} it follows from \eqref{eqn_necessary2} that the corresponding conditions hold for each of the edges $\bbf_q$ ($1 \leq q \leq m$), and thus each term on the right hand side of \eqref{eqn_paths} is defined. The fact that the edges in the right hand side of \eqref{eqn_paths} can be composed in $\Gamma(\gq)$ follows from Lemma~\ref{lem_propertiesofphi}, which also implies:
\begin{equation}
\label{eqn_inversepaths}
\phi(\lambda, \kappa(\bbp, j))^{-1}
=
\phi(\lambda, \kappa(\bbp^{-1}, j)),
\end{equation}
\begin{equation}
\label{eqn_initialandterminalpaths}
\iota \phi(\lambda, \kappa(\bbp, j) )
\equiv
\phi(\lambda, \kappa(\iota\bbp, j) )
\quad \& \quad
\tau \phi(\lambda, \kappa(\bbp, j) )
\equiv
\phi(\lambda, \kappa(\tau\bbp , j) ).
\end{equation}
From now on we shall use $\phi(\bbp)$ as a shorthand for $\phi(1, \kappa(\bbp,\eta))$; 
the introduction of this notation is justified by  
Lemma~\ref{lem_5.10'}. 
The function $\phi(\lambda, \kappa(\bbp, j))$ is only defined for 
certain combinations of $\bbp$, $\lambda$ and $j$.  
The next lemma identifies one important family of such instances.
\begin{lem}
\label{lem_phidefined}
Let $\bbp$ be a path in $\Gamma(\gp)$. If $\iota \bbp $ represents an element of $H$
then
$
\phi(\bbp)
$
is a well-defined path in $\Gamma(\gq)$ from $\phi(\iota \bbp)$ to $\phi(\tau \bbp)$, and $\phi(\bbp)^{-1} = \phi(\bbp^{-1})$.
\end{lem}
\begin{proof}
Recall that $\phi(\bbp) = \phi(1, \kappa(\bbp, \eta))$.
We must check that (i)-(iii) of \eqref{eqn_necessary2} hold. By Remark~\ref{remark_implies} it will suffice just to check (i) and (iii). 

For (i) we use the same argument as in part (i) of Lemma~\ref{lem_5.10'}.
Indeed, since  $e$ represents an element of the right stabilizer of $H$ and $\iota \bbp$ represents an element of $H$ it follows that $\iota \bbp * \eta = \omega$, establishing (i).
For part (iii), $r_{\iota \bbp * j} \equiv r_\omega \equiv h$ and so $H_1 \subseteq Sh$, as required.

The last two claims are consequences of \eqref{eqn_inversepaths} and \eqref{eqn_initialandterminalpaths}.
\end{proof}
The following lemma  tells us that acting on a path $\bbp$ from $\Gamma(\gp)$ and then applying $\phi$ is equivalent to first applying $\phi$ and then acting with appropriate words in $\Gamma(\gq)$.

\begin{lem}\label{lem_expansingphi}
Let $\bbp$ be a path from $\Gamma(\gp)$, $\alpha, \beta \in A^*$, and let $j \in J$, $\lambda \in \Lambda$ be such that
$\alpha \iota \bbp \beta * j \neq 0$ and
$H_{\lambda} \subseteq Sr_{{\alpha \iota \bbp \beta} * j}$.
Then
\begin{align}
& \phi(\lambda, \kappa(\alpha \cdot \bbp \cdot \beta, j) )  \label{eqn_expansingphi1} \\
= & \phi(\lambda, \kappa(\alpha, \iota \bbp \beta * j ))  \cdot 
	\phi(\lambda \cdot \kappa(\alpha, \iota \bbp \beta * j), \kappa(\bbp, \beta * j))  \cdot 
		\phi(\lambda \cdot \kappa(\alpha \iota \bbp, \beta * j), \kappa(\beta, j)).  \label{eqn_expansingphi2}
\end{align}
\end{lem}
\begin{proof}
The proof is straightforward: one first establishes the result for edges using \eqref{eqn_phipath}--\eqref{eqn_phipath2} and Lemmas~\ref{Ruskuc_Lemma5.1}(i) and \ref{Ruskuc_Lemma5.2}(i), and then the result for paths follows straight from \eqref{eqn_paths}.
\end{proof}

\begin{cor}
\label{corol_useful} 
Let $\bbp$ be a path from $\Gamma(\gp)$ and let $\alpha, \beta \in A^*$.
\begin{enumerate}
\item[\emph{(i)}]
If $\alpha \iota \bbp$ represents an element of $H$ then
\begin{equation*}
\phi(\alpha \cdot \bbp) =
\phi(1, \kappa(\alpha, \iota \bbp * \eta)) \cdot
\phi(1 \cdot \kappa(\alpha, \iota \bbp * \eta), \kappa(\bbp, \eta)).
\end{equation*}
\item[\emph{(ii)}]
If $\iota \bbp \beta$ represents an element of $H$ then
\begin{equation*}
\phi(\bbp \cdot \beta) =
\phi(1, \kappa(\bbp, \beta * \eta)) \cdot
\phi(1 \cdot \kappa(\iota \bbp, \beta * \eta), \kappa(\beta, \eta) ).
\end{equation*}
\end{enumerate}
\end{cor}
\begin{proof}
This follows from Lemma~\ref{lem_phidefined} and equation \eqref{eqn_expansingphi1}--\eqref{eqn_expansingphi2} along with the facts:
\[
\phi(1 \cdot \kappa(\alpha \iota \bbp, \eta), \kappa(1, \eta))\equiv 1
\quad
\&
\quad
\phi(1, \kappa(1,\iota \bbp \beta * \eta)) \equiv 1,
\]
which are immediate from the definitions of $\phi$ and $\kappa$.
\end{proof}

The next result shows how the mapping $\phi$ is well behaved with respect to non-overlapping relations.

\begin{lem}
\label{lem_diamonds}
Let $\bbe_1$, $\bbe_2$ be edges from $\Gamma(\gp)$ such that $\iota \bbe_1 \iota \bbe_2$ represents an element of $H$.
Then $\phi([\bbe_1, \bbe_2])$ is null-homotopic in the Squier complex $\mathcal{D}(\gq)$.
\end{lem}
\begin{proof}
Recall that:
\[
[\bbe_1,\bbe_2] =
(\bbe_1 \cdot \iota \bbe_2)\circ
(\tau \bbe_1 \cdot \bbe_2)\circ
(\bbe_1^{-1} \cdot \tau \bbe_2)\circ
(\iota \bbe_1 \cdot \bbe_2^{-1}).
\]
Since $\iota \bbe_2 = \tau \bbe_2$ in $S$ we may define
\[
\bbf_1
 =   \phi(1, \kappa(\bbe_1, \iota \bbe_2 * \eta))
 =   \phi(1, \kappa(\bbe_1, \tau \bbe_2 * \eta)).
\]
Since $\iota \bbe_1 \iota \bbe_2$ represents an element of $H$ it follows that $S r_{\iota \bbe_1 \iota \bbe_2 * \eta} = S r_\omega = Sh \supseteq H_1$ and applying Lemma~\ref{lem_generalfact} we define:
\begin{eqnarray*}
\bbf_2
& = &   \phi(1 \cdot \kappa(\iota\bbe_1, \iota \bbe_2 * \eta), \kappa(\bbe_2, \eta))
 = \phi(1 \cdot \kappa(\tau\bbe_1, \iota \bbe_2 * \eta), \kappa(\bbe_2, \eta)) \\
& = &  \phi(1 \cdot \kappa(\iota\bbe_1, \tau \bbe_2 * \eta), \kappa(\bbe_2, \eta))
 =  \phi(1 \cdot \kappa(\tau\bbe_1, \tau \bbe_2 * \eta), \kappa(\bbe_2, \eta)).
\end{eqnarray*}
So $\bbf_1$ and $\bbf_2$ are edges in the derivation graph $\Gamma(\gq)$.
Then applying Corollary~\ref{corol_useful} and \eqref{eqn_initialandterminal} gives
\begin{align*}
& \; \phantom{=} \;
\; \;
\phi(\bbe_1 \cdot \iota \bbe_2)
\circ
\phi(\tau \bbe_1 \cdot \bbe_2)
\\
& = \; \;
\phi(1, \kappa(\bbe_1, \iota\bbe_2 * \eta)) \cdot
\phi(1 \cdot \kappa(\iota \bbe_1, \iota\bbe_2 * \eta), \kappa(\iota\bbe_2, \eta))
\\
& \hspace{6cm}
\circ
\phi(1, \kappa(\tau\bbe_1, \iota \bbe_2 * \eta)) \cdot
\phi(1 \cdot \kappa(\tau\bbe_1, \iota \bbe_2 * \eta), \kappa(\bbe_2, \eta))
\\
& = \; \;
(\bbf_1 \cdot \iota \bbf_2) \circ (\tau \bbf_1 \cdot \bbf_2)
\\
& \sim_0
(\iota \bbf_1 \cdot \bbf_2) \circ (\bbf_1 \cdot \tau \bbf_2)
\\
& = \; \;
\phi(1, \kappa(\iota\bbe_1, \iota \bbe_2 * \eta)) \cdot
\phi(1 \cdot \kappa(\iota\bbe_1, \iota \bbe_2 * \eta), \kappa(\bbe_2, \eta))
\\
& \hspace{6cm}
\circ
\phi(1, \kappa(\bbe_1, \tau\bbe_2 * \eta)) \cdot
\phi(1 \cdot \kappa(\iota \bbe_1, \tau\bbe_2 * \eta), \kappa(\tau\bbe_2, \eta))
\\
& = \; \;
\phi(\iota \bbe_1 \cdot \bbe_2) \circ \phi(\bbe_1 \cdot \tau \bbe_2).
\end{align*}
Thus $\phi([\bbe_1, \bbe_2])$ is null-homotopic in the Squier complex $\mathcal{D}(\gq)$.
\end{proof}

\begin{cor}
\label{corol_phihomotopy}
Let $\bbp$  and $\bbq$ be paths in $\Gamma(\gp)$ such that $\iota \bbp$ and $\iota \bbq$ both represent elements of $H$. If
$\bbp \sim_0 \bbq$ in $\Gamma(\gp)$ then $\phi(\bbp) \sim_0 \phi(\bbq)$ in $\Gamma(\gq)$.
\end{cor}
\begin{proof}
This follows immediately from Lemmas~\ref{lem_phidefined} and \ref{lem_diamonds}.
\end{proof}

Recall that $X$ is a fixed homotopy base for the monoid $S$ with respect to the presentation $\gp$.
Our aim is to define a homotopy base $Y$ for $\gq$. The set $Y$ will ultimately be given as a union of sets $Y = Y_1 \cup Y_2 \cup Y_3$. We now define the first of these sets which is given my taking images under $\phi$ of closed paths from $X$:
\begin{align*}
Y_1 = & \{ \phi(\lambda \cdot \kappa(\alpha, \iota \bbp \beta * j), \kappa(\bbp, \beta * j))
: \\
& \bbp \in X, \;
\alpha, \beta \in A^*, \;
j \in J, \;
\lambda \in \Lambda, \;
{\alpha \iota \bbp \beta} * j \neq 0, \;
H_{\lambda} \subseteq Sr_{{\alpha \iota \bbp \beta} * j}
\}.
\end{align*}
It follows from Lemma~\ref{lem_expansingphi} that all the elements in this set are well-defined. 
Observe that if $X$, $J$ and $\Lambda$ are all finite, then $Y_1$ is finite. 
\begin{lem}
\label{lem_homotopy1}
Let $\bbp \in X$, $\alpha, \beta \in A^*$, and let $j \in J$, $\lambda \in \Lambda$ such that ${\alpha \iota \bbp \beta} * j \neq 0$ and $H_{\lambda} \subseteq Sr_{{\alpha \iota \bbp \beta} * j}$.
Then in $\Gamma(\gq)$
\[
\phi(\lambda, \kappa(\alpha \cdot \bbp \cdot \beta, j) )
\sim_{Y_1}
1_v
\]
where $v$ is the initial vertex of the path $\phi(\lambda, \kappa(\alpha \cdot \bbp \cdot \beta, j) )$. In particular, if $\bbp \in X$ and $\alpha \iota \bbp \beta$ represents an element of $H$ then
\[
\phi(\alpha \cdot \bbp \cdot \beta) \sim_{Y_1} 1_v.
\]
\end{lem}
\begin{proof}
This is an immediate consequence of \eqref{eqn_expansingphi1}--\eqref{eqn_expansingphi2} and the definition of $Y_1$.
\end{proof}

\subsection*{\boldmath (II) The function $\theta(\bbe)$ (where $\bbe \in \Gamma(\gq)$)}

\noindent We now define a mapping $\theta$ which is defined on the full edge set of $\Gamma(\gq)$, and maps each edge
$\bbe$ of $\Gamma(\gq)$ to a particular path $\theta(\bbe)$ in $\Gamma(\gp)$ from $\pi(\iota \bbe) h$ to $\pi(\tau \bbe) h$.

Begin by defining 
\[
\theta : B^* \rightarrow A^*, \quad w \mapsto \pi (w) h \in A^*.
\]
Then for each relation $(u = v) \in \mathfrak{U}$ fix a path $\bbp [u,v]$ in $\Gamma (\gp)$ from $\pi(u) h$ to $\pi(v) h$. This is possible since $u=v$ holds in the Sch\"{u}tzenberger group $\gG(H)$ if and only if $h \psi (u) = h \psi (v)$ in $S$, which by Lemma~\ref{Ruskuc_Lemma5.3} is true if and only if $\pi(u) h = \pi(v) h$ in $S$.
Also, for every letter $b \in B$ choose and fix a path $\bbp[b]$ in $\Gamma(\gp)$ from $\pi(b)h $ to $h \psi(b)$; this is possible by Lemma~\ref{Ruskuc_Lemma5.3}. For the empty word we set $\bbp[1]$ to be the empty path at $h$. Then for every word $w \in B^+$, with $|w|>1$,  we define $\bbp[w]$ inductively by
\[
\bbp[w] = (\pi(w') \cdot \bbp[b]) \circ (\bbp[w'] \cdot \psi(b) )
\]
where $w \equiv w' b$ and $b\in B$. From the definition we see that $\bbp[w]$ is a path in $\Gamma(\gp)$ with
\begin{equation*}
\iota \bbp[w] = \pi(w) h,
\quad \quad
\tau \bbp[w] = h \psi(w).
\end{equation*}
Then for any positive edge $\bbe = (w_1, u=v, +1, w_2)$ of $\Gamma(\gq)$ we define
\[
\theta (\bbe) =
\pi (w_1) \cdot
[ (\pi(u) \cdot \bbp[w_2]) \circ (\bbp[u,v] \cdot \psi(w_2)) \circ (\pi(v) \cdot \bbp[w_2]^{-1})  ]
\]
and also define  $\theta (\bbe^{-1}) = \theta(\bbe)^{-1}$.
The path $\theta(\bbe)$ is illustrated in Figure~\ref{fig_path_thetaE}, from which the following is clear, so we omit the proof.
\begin{figure}
\begin{center}
\scalebox{0.7}
{
\begin{tikzpicture}
[blackbox/.style={draw, fill=blue!20, rectangle, minimum height=2cm, minimum width=5cm},
dottedbox/.style={draw, rectangle, loosely dashed, minimum height=3.5cm, minimum width=4cm},
clearbox/.style={draw, rectangle, minimum height=2cm, minimum width=4cm}]
\draw[dotted] (2,5) -- (2,7);
\draw[dotted] (3,5) -- (3,7);
\draw[dotted] (4,5) -- (4,7);
\draw[dotted] (2,1) -- (2,3);
\draw[dotted] (3,1) -- (3,3);
\draw[dotted] (4,1) -- (4,3);
\draw[dotted] (7,3) -- (7,5);
\draw[dotted] (8,3) -- (8,5);
\draw[dotted] (9,3) -- (9,5);
\node[clearbox]  (topleft) at (3,6) [label=above:$\pi(w_1 u)$] {};
\node[blackbox] (topright) at (7.5,6) [label=above:$\pi(w_2)h$] {$\mathbb{P}[w_2]$};
\node[blackbox] (midleft) at (3.5,4)  {$\pi(w_1) \cdot \mathbb{P}[u,v]$};
\node[clearbox] (midright) at (8,4)  {};
\node [fill=white] at (8,4) {$\psi(w_2)$};
\node[clearbox] (bottomleft) at (3,2) [label=below:$\pi(w_1 v)$] {};
\node[blackbox] (bottomright) at (7.5,2) [label=below:$\pi(w_2)h$] {$\mathbb{P}[w_2]^{-1}$};
\draw[very thick] (5,5) -- node[above]  {$h$} (6,5);
\draw[very thick] (5,3) -- node[below]  {$h$} (6,3);
\end{tikzpicture}
}
\end{center}
\caption{The path $\theta(\mathbb{E})$ where $\mathbb{E} = (w_1, u=v, +1, w_2)$.}
\label{fig_path_thetaE}
\end{figure}
\begin{lem}
\label{lem_thetaimage}
For every edge $\bbe$ of $\Gamma(\gq)$, $\theta(\bbe)$ is a path in $\Gamma(\gp)$ from $\pi(\iota \bbe) h$ to $\pi(\tau \bbe) h$.
\end{lem}

\subsection*{\boldmath (III) The paths $\Lambda_w$ (where $w \in B^*$).}

\noindent For every $w \in B^*$, we shall now define a path $\Lambda_w$ in $\Gamma (\gq)$ which will satisfy:
\[
\iota \Lambda_w \equiv w, \quad
\tau \Lambda_w \equiv \phi(1, \kappa(\pi(w)h, \eta) ) \equiv \phi(\pi(w)h).
\]
The definition is essentially extracted from the proof of Lemma~\ref{lem_5.9'}, where it was shown that
\[
\phi(1, \kappa( \pi(w)h, \eta)) = w
\]
is a consequence of the presentation $\gq$. To facilitate the definition of $\Lambda_w$ we introduce a little more notation.

\begin{defn}[Arrow notation]
Let $\gp$ be a monoid presentation and let $\Gamma(\gp)$ denote its derivation graph.
For any two paths $\bbp$, $\bbq$ in $\Gamma(\gp)$ we define:
\begin{enumerate}[(i)]
\item $\bbp \downarrow \bbq = (\bbp \cdot \iota \bbq) \circ (\tau \bbp \cdot \bbq) $;
\item $\bbp \uparrow \bbq = (\iota \bbp \cdot \bbq) \circ (\bbp \cdot \tau \bbq) $.
\end{enumerate}
\end{defn}

The following observations are easily verified, and so the proofs are omitted.

\begin{lem}
\label{lem_behaviour}
For paths $\bbp, \bbq, \bba, \bbb$ in $\Gamma(\gp)$ we have:
\begin{enumerate}
\item[\emph{(i)}] $\bbp \da (\bba \circ \bbb) = (\bbp \da \bba) \circ (\tau \bbp \cdot \bbb)$;
\item[\emph{(ii)}] $(\bba \circ \bbb) \da \bbp = (\bba \cdot \iota \bbp) \circ (\bbb \da \bbp)$;
\item[\emph{(iii)}] $\bbp \da (\bba \circ \bbb) \da \bbq =
[(\bbp \da \bba) \cdot \iota \bbq] \circ
[\tau \bbp \cdot (\bbb \da \bbq)]
$;
\item[\emph{(iv)}] $(\bbp \da \bbq )^{-1}= \bbp^{-1} \ua \bbq^{-1}$;
\item[\emph{(v)}] $\bbp \da \bbq \sim_0 \bbp \ua \bbq$.
\end{enumerate}
\end{lem}
Next we define
\[
\bbh
= (1, \; 1 = \phi(1, \kappa(h,\eta)), \; +1, 1 )
= \bba_{1 = \phi(1, \kappa(h,\eta))},
\]
an elementary edge in $\Gamma(\gq)$ corresponding to a relation of the form (R4), and
for each $b \in B$ define
\[
\bbb_b = (1, \; b = \phi(1, \kappa(\pi(b),\omega)), \; +1, 1 )
=
\bba_{b = \phi(1, \kappa(\pi(b),\omega))}
,
\]
an elementary edge in $\Gamma(\gq)$ corresponding to a relation of the form (R3).  
Set  $\Lambda_1 = \bbh$, and then for each $w \equiv b_1 \cdots b_k \in B^+$ define:
\[
\Lambda_w = \bbb_{b_1} \da \bbb_{b_2} \da \cdots \da \bbb_{b_k} \da \bbh,
\]
and
\[
\Lambda_w' = \bbb_{b_1} \da \bbb_{b_2} \da \cdots \da \bbb_{b_k}.
\]
\begin{lem}
\label{lem_Lambda}
For all $w \in B^*$, $\Lambda_w$ is a path in $\Gamma(\gq)$ from $w$ to $\phi(1, \kappa(\pi(w)h, \eta) )$.
\end{lem}
\begin{proof}
When $w$ is empty, this is immediate from the definition of $\Lambda_1 = \bbh$, since $\pi(1)\equiv 1$. Now suppose  $w \equiv b_1 b_2 \cdots b_k \in B^+$. Then
from the definition, we have
\[
\iota \Lambda_w = b_1 b_2 \cdots b_k 1 \equiv w,
\]
and, arguing as in the proof of Lemma~\ref{lem_5.9'}:
\begin{align*}
\tau \Lambda_w
& \equiv 
\phi(1, \kappa(\pi(b_1),\omega))
\phi(1, \kappa(\pi(b_2),\omega))
\cdots
\phi(1, \kappa(\pi(b_k),\omega))
\phi(1, \kappa(h,\eta))
\\
& \equiv
\phi(1, \kappa(\pi(b_1 b_2 b_3 \cdots b_k)h,\eta))
\\
& \equiv
\phi(1, \kappa(\pi(w)h,\eta)). \qedhere
\end{align*}
\end{proof}

\subsection*{(IV)-(V) An infinite trivialiser}

Let $Z$ denote the following (infinite) set of closed paths from the graph $\Gamma(\gq)$
\[
Z = \{
\bbe \circ \; \Lambda_{\tau \bbe} \circ \phi (\theta (\bbe))^{-1} \circ \Lambda_{\iota \bbe}^{-1}
\ \mbox{for $\bbe$ in $\Gamma(\gq)$}
 \}.
\]
Note that by Lemma~\ref{lem_thetaimage}, $\theta(\bbe)$ is a path from $\pi(\iota \bbe)h$ to $\pi(\tau \bbe)h$, which are both words representing an element of $H$, and so by Lemma~\ref{lem_phidefined}, $\phi(\theta(\bbe))$ is a path in $\Gamma(\bbq)$ from $\phi(\pi(\iota \bbe)h)$ to  $\phi(\pi(\tau \bbe)h)$. From this it is then easy to see that $Z$ is a set of closed paths in $\Gamma(\mathcal{Q})$. 

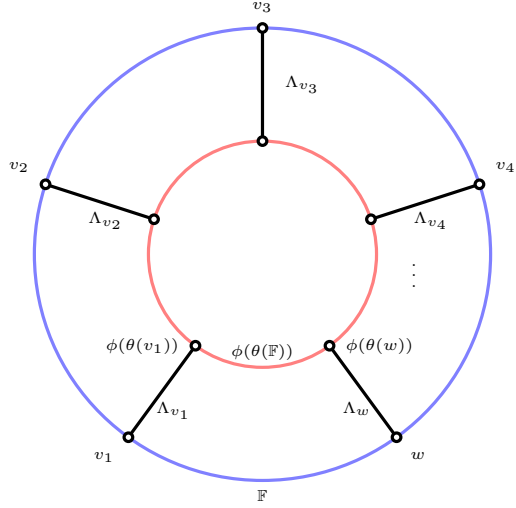
\begin{figure}
\begin{center}
\begin{tikzpicture}
[
very thick,
vertex/.style={circle,draw=black, fill=white, inner sep = 0.4mm}
]
\draw [blue!50] (0,0) circle (3cm);
\draw [red!50] (0,0) circle (1.5cm);
\draw (90:1.5cm) -- node[right] {\tiny \; $\Lambda_{v_3}$}  (90:3cm);
\draw (162:1.5cm) -- node[below] {\tiny \; $\Lambda_{v_2}$} (162:3cm);
\draw (234:1.5cm) -- node[below] {\tiny \; \; $\Lambda_{v_1}$} (234:3cm);
\draw (306:1.5cm) -- node[below] {\tiny $\Lambda_w$ \;} (306:3cm);
\draw (18:1.5cm) -- node[below] {\tiny \; $\Lambda_{v_4}$}  (18:3cm);
\node (v1) [vertex,label={90:{\tiny $v_3$}}] at (90:3cm) {};
\node (v2) [vertex,label={162:{\tiny $v_2$}}] at (162:3cm) {};
\node (v3) [vertex,label={234:{\tiny $v_1$}}] at (234:3cm) {};
\node (v4) [vertex,label={306:{\tiny $w$}}] at (306:3cm) {};
\node (v5) [vertex,label={18:{\tiny $v_4$}}] at (18:3cm) {};
\node (w1) [vertex,label={90:{}}] at (90:1.5cm) {};
\node (w2) [vertex,label={162:{}}] at (162:1.5cm) {};
\node (w3) [vertex,label={180:{\tiny $\phi(\theta(v_1))$}}] at (234:1.5cm) {};
\node (w4) [vertex,label={0:{\tiny $\phi(\theta(w))$}}] at (306:1.5cm) {};
\node (w5) [vertex,label={18:{}}] at (18:1.5cm) {};
\node at (355:2.0cm) {\tiny $\vdots$};
\node at (270:1.3cm) {\tiny $\phi(\theta(\bbf))$};
\node at (270:3.2cm) {\tiny $\bbf$};
\end{tikzpicture}
\end{center}
\caption{Trivialising paths in $\Gamma(\mathcal{Q})$.}
\label{fig_infinite_triv}
\end{figure}

\begin{lem}
\label{lem_infinitetriv}
The set $Z \cup Y_1$ is a homotopy base for $\Gamma(\gq)$.
\end{lem}
\begin{proof}
Let $\bbf$ be a closed path in $\Gamma(\gq)$ based at a vertex $w \in B^*$. We must show that $\bbf \sim_{Z \cup Y_1} 1_w$, the empty path at $w$.
For each edge $\bbe$ in the closed path $\bbf$ we have:
\[
\bbe \sim_Z \Lambda_{\iota \bbe} \circ \phi(\theta(\bbe)) \circ \Lambda_{\tau \bbe}^{-1},
\]
from which it easily follows that
\[
\bbf  \sim_Z \Lambda_w \circ \phi(\theta(\bbf)) \circ \Lambda_w^{-1}.
\]
(See Figure~\ref{fig_infinite_triv} for an illustration of this: each of the regions that lie between the inner circuit $\phi(\theta(\bbf))$ and the outer circuit $\bbf$, belongs to the set $Z$.)

Denote $\theta(\bbf)$ by $\bbp$. Now $\bbp$ is a closed path in $\Gamma(\gp)$ and $\iota \bbp$ is a word representing an element of $H$
(since $\theta(w) = \pi(w) h$ represents an element of $H$). Since $X$ is a homotopy base for $\gp$ it follows from Lemma~\ref{lem_homgen} that  there exist $v_i, w_i \in A^*$, paths $\bbp_i$ in $\Gamma$ and $\bbq_i \in X \cup X^{-1}$, for $i=1,\ldots,n$, \ $n \geq 0$, such that:
\[
\bbp
\sim_0
\bbp_1^{-1} \circ (v_1 \bbq_1 w_1) \circ \bbp_1 \circ \cdots \circ \bbp_n^{-1} \circ (v_n \bbq_n w_n) \circ \bbp_n
\]
in $\Gamma(\gp)$. Applying Corollary~\ref{corol_phihomotopy} then gives
\begin{eqnarray*}
\phi(\bbp)
& \sim_0 &
\phi(\bbp_1^{-1} \circ (v_1 \bbq_1 w_1) \circ \bbp_1 \circ \cdots \circ \bbp_n^{-1} \circ (v_n \bbq_n w_n) \circ \bbp_n)
\\
& = &
\phi(\bbp_1^{-1}) \circ \phi(v_1 \bbq_1 w_1) \circ \phi(\bbp_1) \circ \cdots \circ \phi(\bbp_n^{-1}) \circ \phi(v_n \bbq_n w_n) \circ \phi(\bbp_n).
\end{eqnarray*}
By Lemma~\ref{lem_homotopy1} each term $\phi(v_j \bbq_j w_j)$ is $\sim_{Y_1}$-homotopic to the empty path $1_{\iota \phi(\bbp_j)}$, which along with \eqref{eqn_inversepaths} shows that $\phi(\theta(\bbf)) = \phi(\bbp) \sim_{Y_1} 1_{\phi(\theta(w))}$ in $\Gamma(\gq)$. In conclusion
\[
\bbf
\sim_Z \Lambda_w \circ \phi(\theta(\bbf)) \circ \Lambda_w^{-1}
\sim_{Y_1} \Lambda_w \circ \Lambda_w^{-1} \sim_0 1_w
\]
completing the proof of the lemma.
\end{proof}

As already observed above, $Y_1$ is finite provided $X$, $J$ and $\Lambda$ are all finite. Our aim now is to replace $Z$ by a set $Y_2 \cup Y_3$ with the property that $Y_2 \cup Y_3$ is finite whenever
$X$, $J$ and $\Lambda$ are all finite.

\subsection*{\boldmath (VI) Replacing $Z$ by $Y_2 \cup Y_3$}

\noindent Recall that $h$ is a fixed word in $A^*$ representing an element of $H$. Define
\[
W = \{
\phi (1, \kappa(h,j)): j \in J, \; h * j \neq 0
\}
\subseteq B^*
\]
noting that this set of words is finite provided the index set $J$ is finite. Now, each word from $W$ represents an element of the Sch\"{u}tzenberger group $\mathcal{G}(H)$. So, for every $y \in W$ we may choose and fix a word $\hat{y} \in B^*$ representing the inverse of $y$ in the group $\mathcal{G}(H)$. Next choose and fix paths:
\begin{itemize}
\item $\bbd_y$ in $\Gamma(\mathcal{Q})$ from $y \hat{y}$ to the empty word, and
\item $\bbd_y^*$ in $\Gamma(\mathcal{Q})$ from $\hat{y} y$ to the empty word.
\end{itemize}
For each $y \in W$, define 
\[
\bby_y = (\bbd_y^{-1} \cdot y) \circ (y \cdot \bbd_y^*)
\]
which is a closed path in $\Gamma(\mathcal{Q})$ with $\iota \bby_y \equiv \tau \bby_y \equiv y$. The path $\bby_y$ is illustrated in Figure~\ref{fig_Yy}.
Set $Y_2 = \{ \bby_y: y \in W  \}$ noting that $Y_2$ is finite if $J$ is finite. 
\begin{figure}
\begin{center}
\scalebox{0.7}
{
\begin{tikzpicture}
[blackbox/.style={draw, fill=blue!20, rectangle, minimum height=1cm, minimum width=2cm},
dottedbox/.style={draw, rectangle, loosely dashed, minimum height=3.5cm, minimum width=4cm}]
\node at (1.5,1.25) [dottedbox] {};
\node at (1,2) [blackbox] {$\mathbb{D}_y^{-1}$};
\node at (2,0.5) [blackbox] {$\mathbb{D}_y^*$};
\draw (.333,-.5) -- (.333,1.5);
\draw (.5,1.5)  -- (.5,-.5) node[anchor=north]  {$y$};
\draw (.666,-.5) -- (.666,1.5);
\draw (1.333,1) -- (1.333,1.5);
\draw (1.5,1) -- (1.5,1.5);
\draw (1.666,1) -- (1.666,1.5);
\draw (2.333,1) -- (2.333,3);
\draw (2.5,1) -- (2.5,3) node[anchor=south]  {$y$};
\draw (2.666,1) -- (2.666,3);
\node at (1.8,1.25) {$\hat{y}$};
\end{tikzpicture}
}
\end{center}
\caption{The path $\mathbb{Y}_y$.}
\label{fig_Yy}
\end{figure}
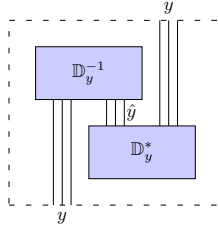
The following lemma demonstrates how we intend to use the paths $Y_2$.
\begin{lem}
\label{lem_form}
Let $\bbp$ and $\bbf$ be paths in $\Gamma (\gq)$ such that there exist $y \in W$ and $\alpha, \beta, \gamma \in B^*$ with $\tau \bbp \equiv y \gamma$, $\iota \bbf \equiv \alpha y$, and $\tau \bbf \equiv \beta y$.
Then
\[
(\alpha \cdot \bbp) \circ (\bbf \cdot \gamma) \circ (\beta \cdot \bbp^{-1}) \sim_{Y_2}
[
(\alpha \cdot {\bbd_y}^{-1}) \circ
(\bbf \cdot \hat{y}) \circ
(\beta \cdot \bbd_y)
] \cdot \iota \bbp.
\]
\end{lem}
\begin{proof}
The key steps in the proof are laid out in Figures~\ref{fig_step0}--\ref{fig_step3}. 
First we introduce the closed paths $\bby_y$ and $\bby_y^{-1}$ (Step 1). 
Then we cancel $\bbd_y^*$ with its inverse (Step 2). This then allows $\bbp$ to cancel with $\bbp^{-1}$ (Step 3).
Symbolically the proof takes the following form:
\begin{align*}
& \;\phantom{\sim_{Y_2}}\; & & (\alpha \cdot \bbp) \circ (\bbf \cdot \gamma) \circ (\beta \cdot \bbp^{-1}) & \\
& \sim_{Y_2} & &
(\alpha \cdot \bbp) \circ (\alpha \cdot \bby_y \cdot \gamma) \circ (\bbf \cdot \gamma) \circ (\beta \cdot \bby_y^{-1} \cdot \gamma) \circ (\beta \cdot \bbp^{-1})  \quad \quad \quad \quad \quad  \mbox{(Step 1)}
&  \\
& = & &
(\alpha \cdot \bbp) \circ
(\alpha \cdot \bbd_y^{-1} \cdot y \gamma) \circ
(\alpha y \cdot \bbd_y^* \cdot \gamma)
\circ
\underline{
(\bbf \cdot \gamma) \circ
(\beta y \cdot {\bbd_y^*}^{-1} \cdot \gamma)
}
\circ
(\beta \cdot \bbd_y \cdot y\gamma)
\circ (\beta \cdot \bbp^{-1}) \\
& \sim_0 & &
(\alpha \cdot \bbp) \circ
(\alpha \cdot \bbd_y^{-1} \cdot y \gamma) \circ
\underline{
(\alpha y \cdot \bbd_y^* \cdot \gamma)
\circ
(\alpha y \cdot {\bbd_y^*}^{-1} \cdot \gamma)
}
\circ
(\bbf \cdot \hat{y} y \gamma)
\circ
(\beta \cdot \bbd_y \cdot y\gamma)
\circ (\beta \cdot \bbp^{-1}) \\
& \sim_0 & &
(\alpha \cdot \bbp) \circ
\underline{
(\alpha \cdot \bbd_y^{-1} \cdot y \gamma)
\circ
(\bbf \cdot \hat{y} y \gamma)
\circ
(\beta \cdot \bbd_y \cdot y\gamma)
}
\circ (\beta \cdot \bbp^{-1}) \quad \quad \quad \mbox{(Step 2)} &  \\
& = & & 
\underline{
(\alpha \cdot \bbp) \circ \;
([(\alpha \cdot \bbd_y^{-1}) \circ (\bbf \cdot \hat{y}) \circ (\beta \cdot \bbd_y)] \cdot y\gamma) \;
}
\circ
(\beta \cdot \bbp^{-1})
\\
& \sim_0 & &  
([(\alpha \cdot \bbd_y^{-1}) \circ (\bbf \cdot \hat{y}) \circ (\beta \cdot \bbd_y)] \cdot \iota \bbp ) \;
\circ
\underline{
(\beta \cdot \bbp)
\circ (\beta \cdot \bbp^{-1})
} \\
& = & & 
[
(\alpha \cdot {\bbd_y}^{-1}) \circ
(\bbf \cdot \hat{y}) \circ
(\beta \cdot \bbd_y)
] \cdot \iota \bbp \quad \quad \quad  \quad \quad \quad   \quad \quad \quad \quad   \quad \quad  \quad \; \;  \mbox{(Step 3)}
& 
\end{align*}
as required (we have underlined the places where the homotopy relations are being applied in this argument). 
\end{proof}
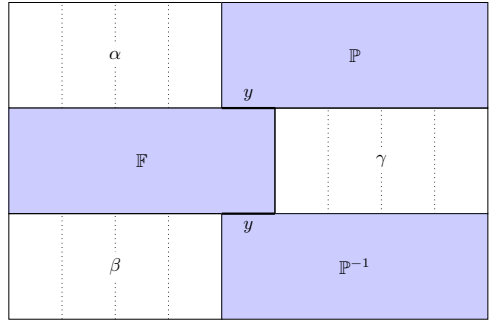
\begin{figure}
\begin{center}
\scalebox{0.7}
{
\begin{tikzpicture}
[blackbox/.style={draw, fill=blue!20, rectangle, minimum height=2cm, minimum width=5cm},
dottedbox/.style={draw, rectangle, loosely dashed, minimum height=3.5cm, minimum width=4cm},
clearbox/.style={draw, rectangle, minimum height=2cm, minimum width=4cm}]
\draw[dotted] (2,5) -- (2,7);
\draw[dotted] (3,5) -- (3,7);
\draw[dotted] (4,5) -- (4,7);
\draw[dotted] (2,1) -- (2,3);
\draw[dotted] (3,1) -- (3,3);
\draw[dotted] (4,1) -- (4,3);
\draw[dotted] (7,3) -- (7,5);
\draw[dotted] (8,3) -- (8,5);
\draw[dotted] (9,3) -- (9,5);
\node[clearbox]  (topleft) at (3,6) {};
\node[blackbox] (topright) at (7.5,6) {$\mathbb{P}$};
\node[blackbox] (midleft) at (3.5,4)  {$\mathbb{F}$};
\node[clearbox] (midright) at (8,4)  {};
\node [fill=white] at (8,4) {$\gamma$};
\node [fill=white] at (3,6) {$\alpha$};
\node [fill=white] at (3,2) {$\beta$};
\node[clearbox] (bottomleft) at (3,2) {};
\node[blackbox] (bottomright) at (7.5,2) {$\mathbb{P}^{-1}$};
\draw[very thick] (5,5) -- node[above]  {$y$} (6,5);
\draw[very thick] (5,3) -- node[below]  {$y$} (6,3);
\end{tikzpicture}
}
\end{center}
\caption{Proof of Lemma~\ref{lem_form}: Initial configuration.}
\label{fig_step0}
\end{figure}

\begin{figure}
\begin{center}
\begin{tabular}{cc}
\scalebox{0.7}
{
\begin{tikzpicture}
[blackbox/.style={draw, fill=blue!20, rectangle, minimum height=2cm, minimum width=5cm},
bigbox/.style={draw, rectangle, minimum height=11cm, minimum width=9cm},
clearbox/.style={draw, rectangle, minimum height=2cm, minimum width=4cm},
tinybox/.style={draw, fill=blue!20, rectangle, minimum height=.6cm, minimum width=1.1cm}
]
\draw[thick] (5.5,8.5) --  (5.5,9);
\node at (5.6,8.75) {\begin{tiny} $\hat{y}$ \end{tiny}};
\draw[thick, rounded corners=20pt] (5.5,10) -- (6,10)  --  (6,8.5);
\draw[thick, rounded corners=20pt] (5.5,7.5) -- (5,7.5) -- (5,9);
\draw[thick] (5.5,4) --  (5.5,4.5);
\node at (5.6,4.25) {\begin{tiny} $\hat{y}$ \end{tiny}};
\draw[thick, rounded corners=20pt] (5.5,5.5) -- (5,5.5)  --  (5,4);
\draw[thick, rounded corners=20pt] (5.5,3) -- (6,3) -- (6,4.5);
\node[bigbox] (border) at (5.5,6.5)  {};
\node[tinybox] (D1) at (5.25,9.25)  {\begin{tiny} $\mathbb{D}_y^{-1}$ \end{tiny}};
\node[tinybox] (D2) at (5.75,8.25)  {\begin{tiny} $\mathbb{D}_y^*$ \end{tiny}};
\node[tinybox] (D3) at (5.75,4.75)  {\begin{tiny} ${\mathbb{D}_y^*}^{-1}$ \end{tiny}};
\node[tinybox] (D4) at (5.25,3.75)  {\begin{tiny} $\mathbb{D}_y$ \end{tiny}};
\draw[dotted] (2,7.5) -- (2,12);
\draw[dotted] (3,7.5) -- (3,12);
\draw[dotted] (4,7.5) -- (4,12);
\draw[dotted] (2,1) -- (2,5.5);
\draw[dotted] (3,1) -- (3,5.5);
\draw[dotted] (4,1) -- (4,5.5);
\draw[dotted] (7,3) -- (7,10);
\draw[dotted] (8,3) -- (8,10);
\draw[dotted] (9,3) -- (9,10);
\node[clearbox]  (topleft) at (3,11) {};
\node[blackbox] (topright) at (7.5,11) {$\mathbb{P}$};
\node[blackbox] (midleft) at (3.5,6.5)  {$\mathbb{F}$};
\node[clearbox] (midright) at (8,6.5)  {};
\node [fill=white] at (8,6.5) {$\gamma$};
\node [fill=white] at (8,8.75) {$\gamma$};
\node [fill=white] at (8,4.25) {$\gamma$};
\node [fill=white] at (3,11) {$\alpha$};
\node [fill=white] at (3,8.75) {$\alpha$};
\node [fill=white] at (3,2) {$\beta$};
\node [fill=white] at (3,4.25) {$\beta$};
\node[clearbox] (bottomleft) at (3,2) {};
\node[blackbox] (bottomright) at (7.5,2) {$\mathbb{P}^{-1}$};
\draw[very thick] (5,7.5) -- node[below]  {$y$} (6,7.5);
\draw[very thick] (5,5.5) -- node[above]  {$y$} (6,5.5);
\draw[very thick] (5,10) -- node[above]  {$y$} (6,10);
\draw[very thick] (5,3) -- node[below]  {$y$} (6,3);
\end{tikzpicture}
}
&
\scalebox{0.7}
{
\begin{tikzpicture}
[blackbox/.style={draw, fill=blue!20, rectangle, minimum height=2cm, minimum width=5cm},
bigbox/.style={draw, rectangle, minimum height=11cm, minimum width=9cm},
clearbox/.style={draw, rectangle, minimum height=2cm, minimum width=4cm},
tinybox/.style={draw, fill=blue!20, rectangle, minimum height=.6cm, minimum width=1.1cm}
]
\draw[thick, rounded corners=25pt] (5.5,9) -- (6.2,9) -- (6.25,7.5);
\node at (5.6,8.75) {\begin{tiny} $\hat{y}$ \end{tiny}};
\draw[thick, rounded corners=20pt] (5.5,10) -- (6.5,9)  --  (6.5,7.5);
\draw[thick, rounded corners=20pt] (5.5,7.5) -- (5,7.5) -- (5,9);
\draw[thick] (6.25,7.5) -- (6.25,5.5);
\draw[thick] (6.5,7.5) -- (6.5,5.5);
\draw[thick, rounded corners=25pt] (5.5,4) -- (6.2,4) -- (6.25,5.5);
\node at (5.6,4.25) {\begin{tiny} $\hat{y}$ \end{tiny}};
\draw[thick, rounded corners=20pt] (5.5,5.5) -- (5,5.5)  --  (5,4);
\draw[thick, rounded corners=20pt] (5.5,3) -- (6.5,4) -- (6.5,5.5);
\node[bigbox] (border) at (5.5,6.5)  {};
\node[tinybox] (D1) at (5.25,9.25)  {\begin{tiny} $\mathbb{D}_y^{-1}$ \end{tiny}};
\node[tinybox] (D4) at (5.25,3.75)  {\begin{tiny} $\mathbb{D}_y$ \end{tiny}};
\draw[dotted] (2,7.5) -- (2,12);
\draw[dotted] (3,7.5) -- (3,12);
\draw[dotted] (4,7.5) -- (4,12);
\draw[dotted] (2,1) -- (2,5.5);
\draw[dotted] (3,1) -- (3,5.5);
\draw[dotted] (4,1) -- (4,5.5);
\draw[dotted] (7,3) -- (7,10);
\draw[dotted] (8,3) -- (8,10);
\draw[dotted] (9,3) -- (9,10);
\node[clearbox]  (topleft) at (3,11) {};
\node[blackbox] (topright) at (7.5,11) {$\mathbb{P}$};
\node[blackbox] (midleft) at (3.5,6.5)  {$\mathbb{F}$};
\node[clearbox] (midright) at (8,6.5)  {};
\node [fill=white] at (8,6.5) {$\gamma$};
\node [fill=white] at (8,8.75) {$\gamma$};
\node [fill=white] at (8,4.25) {$\gamma$};
\node [fill=white] at (3,11) {$\alpha$};
\node [fill=white] at (3,8.75) {$\alpha$};
\node [fill=white] at (3,2) {$\beta$};
\node [fill=white] at (3,4.25) {$\beta$};
\node[clearbox] (bottomleft) at (3,2) {};
\node[blackbox] (bottomright) at (7.5,2) {$\mathbb{P}^{-1}$};
\draw[very thick] (5,7.5) -- node[below]  {$y$} (6,7.5);
\draw[very thick] (5,5.5) -- node[above]  {$y$} (6,5.5);
\draw[very thick] (5,10) -- node[above]  {$y$} (6,10);
\draw[very thick] (5,3) -- node[below]  {$y$} (6,3);
\end{tikzpicture}
}
\end{tabular}
\end{center}
\caption{Proof of Lemma~\ref{lem_form}: Steps 1 and 2.}
\label{fig_step2}
\end{figure}
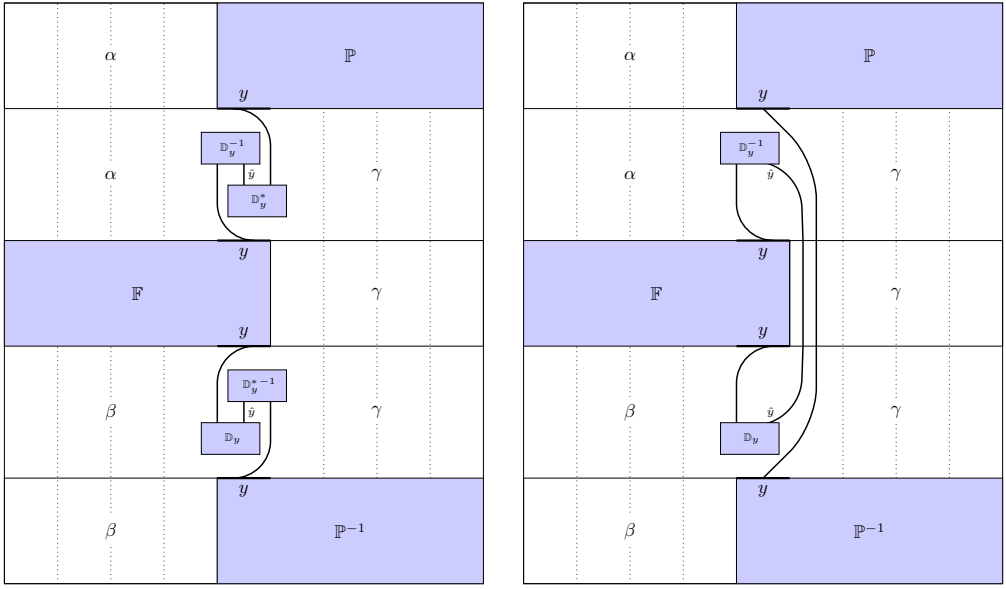

To simplify expressions appearing below we introduce the notation $\phi_\omega(\pi(u)) = \phi(1, \kappa(\pi(u), \omega))$.
Note that by Lemmas~\ref{Ruskuc_Lemma5.1}(i), ~\ref{Ruskuc_Lemma5.2}(i) and \ref{Ruskuc_Lemma5.4} we have
\begin{equation}
\label{omega_eq}
\phi_\omega(\pi(w_1w_2))=\phi_\omega(\pi(w_1))\phi_\omega(\pi(w_2)),
\end{equation}
for all $w_1,w_2\in B^*$.

The following lemma is crucial for the proof of our main result. It provides an important decomposition for $\phi(\theta(\bbe))$ where $\bbe$ is an arbitrary edge of $\Gamma(\mathcal{Q})$. 

\begin{lem}
\label{lem_reducing}
Let $\bbe = (w_1, (u=v), \epsilon, w_2) \in \Gamma(\gq)$.
Then $\phi (\theta (\bbe))$ is $\sim_{Y_2}$-homotopic to
\[
\phi_\omega(\pi(w_1))
\cdot
\;
[
(\phi_\omega(\pi(u)) \cdot \bbd_y^{-1}) \circ
(\phi(1,\kappa(\bbp[u,v]^\epsilon, \psi(w_2) * \eta )) \cdot \hat{y}) \circ
(\phi_\omega(\pi(v)) \cdot \bbd_y)
]
\;
\cdot
\phi(\pi(w_2)h)
\]
where $y \equiv \phi(1, \kappa(h, \psi(w_2) * \eta)$.
\end{lem}
%
%
%
\begin{figure}
\begin{center}
\scalebox{0.7}
{
\begin{tikzpicture}
[blackbox/.style={draw, fill=blue!20, rectangle, minimum height=2cm, minimum width=5cm},
bigboxleft/.style={draw, rectangle, minimum height=7cm, minimum width=5.5cm},
bigboxright/.style={draw, rectangle, minimum height=7cm, minimum width=2.5cm},
clearbox/.style={draw, rectangle, minimum height=2cm, minimum width=4cm},
tinybox/.style={draw, fill=blue!20, rectangle, minimum height=.6cm, minimum width=1.1cm}
]
\node[bigboxleft] at (3.75,6.5)  {};
\node[bigboxright] at (8.5,6.5)  {};
\draw[thick, rounded corners=25pt] (5.5,9) -- (6.2,9) -- (6.25,7.5);
\node at (5.6,8.75) {\begin{tiny} $\hat{y}$ \end{tiny}};
\draw[thick, rounded corners=20pt] (5.5,7.5) -- (5,7.5) -- (5,9);
\draw[thick] (6.25,7.5) -- (6.25,5.5);
\draw[thick, rounded corners=25pt] (5.5,4) -- (6.2,4) -- (6.25,5.5);
\node at (5.6,4.25) {\begin{tiny} $\hat{y}$ \end{tiny}};
\draw[thick, rounded corners=20pt] (5.5,5.5) -- (5,5.5)  --  (5,4);
\node[tinybox] (D1) at (5.25,9.25)  {\begin{tiny} $\mathbb{D}_y^{-1}$ \end{tiny}};
\node[tinybox] (D4) at (5.25,3.75)  {\begin{tiny} $\mathbb{D}_y$ \end{tiny}};
\draw[dotted] (2,7.5) -- (2,10);
\draw[dotted] (3,7.5) -- (3,10);
\draw[dotted] (4,7.5) -- (4,10);
\draw[dotted] (2,3) -- (2,5.5);
\draw[dotted] (3,3) -- (3,5.5);
\draw[dotted] (4,3) -- (4,5.5);
\draw[dotted] (7.5,3) -- (7.5,10);
\draw[dotted] (8.5,3) -- (8.5,10);
\draw[dotted] (9.5,3) -- (9.5,10);
\node[blackbox] (midleft) at (3.5,6.5)  {$\mathbb{F}$};
\node [fill=white] at (8.5,6.5) {$\iota \mathbb{P}$};
\node [fill=white] at (8.5,8.75) {$\iota \mathbb{P}$};
\node [fill=white] at (8.5,4.25) {$\iota \mathbb{P}$};
\node [fill=white] at (3,8.75) {$\alpha$};
\node [fill=white] at (3,4.25) {$\beta$};
\draw[very thick] (5,7.5) -- node[below]  {$y$} (6,7.5);
\draw[very thick] (5,5.5) -- node[above]  {$y$} (6,5.5);
\end{tikzpicture}
}
\end{center}
\caption{Proof of Lemma~\ref{lem_form}: Step 3.}
\label{fig_step3}
\end{figure}
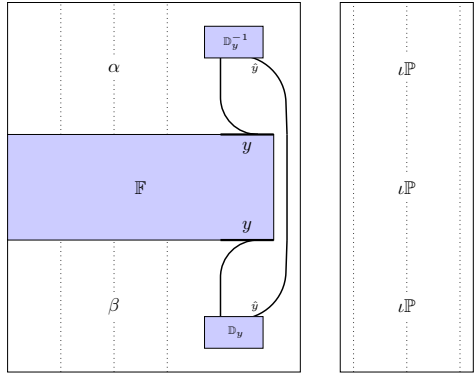
\begin{proof}
Let $\bbe = (w_1, (u=v), \epsilon, w_2)$ be an arbitrary edge of $\Gamma(\gq)$. Let us suppose throughout the argument that $\epsilon = +1$. The proof in the case $\epsilon=-1$ is analogous. 

The idea of the proof is as follows. We shall show that $\phi(\theta(\bbe))$ has the form illustrated in Figure~\ref{fig_phi_theta_E}. Comparing with Figure~\ref{fig_step0} we see that $\phi(\theta(\bbe))$ satisfies the hypotheses of Lemma~\ref{lem_form}, which we then apply to prove the lemma.  

Firstly, by definition $\theta(\bbe) = \bbq_1 \circ \bbq_2 \circ \bbq_3$ where:
$$
\bbq_1 = \pi(w_1) \pi(u) \cdot \bbp[w_2],
\; \quad \;
\bbq_2 = \pi(w_1) \cdot \bbp[u,v] \cdot \psi(w_2),
$$
$$
\bbq_3 = \pi(w_1)\pi(v) \cdot \bbp[w_2]^{-1}.
$$
Now consider $\phi(\theta(\bbe))$. By definition of $\theta$, all the vertices in each of the paths $\bbq_1$, $\bbq_2$ and $\bbq_3$ are words representing an element of $H$, and so by Lemma~\ref{lem_phidefined} each of $\phi(\bbq_1)$, $\phi(\bbq_2)$ and $\phi(\bbq_3)$ is defined. These paths then decompose as follows:
\begin{figure}
\begin{center}
\scalebox{0.8}
{
\begin{tikzpicture}
[blackbox/.style={draw, fill=blue!20, rectangle, minimum height=2cm, minimum width=11cm},
dottedbox/.style={draw, rectangle, loosely dashed, minimum height=3.5cm, minimum width=8cm},
clearbox/.style={draw, rectangle, minimum height=2cm, minimum width=7cm}]
\draw[dotted] (2,5) -- (2,7);
\draw[dotted] (3,5) -- (3,7);
\draw[dotted] (4,5) -- (4,7);
\draw[dotted] (5,5) -- (5,7);
\draw[dotted] (6,5) -- (6,7);
\draw[dotted] (7,5) -- (7,7);
\draw[dotted] (2,1) -- (2,3);
\draw[dotted] (3,1) -- (3,3);
\draw[dotted] (4,1) -- (4,3);
\draw[dotted] (5,1) -- (5,3);
\draw[dotted] (6,1) -- (6,3);
\draw[dotted] (7,1) -- (7,3);
\draw[dotted] (13,3) -- (13,5);
\draw[dotted] (14,3) -- (14,5);
\draw[dotted] (15,3) -- (15,5);
\draw[dotted] (16,3) -- (16,5);
\draw[dotted] (17,3) -- (17,5);
\draw[dotted] (18,3) -- (18,5);
\draw[dotted] (19,3) -- (19,5);
\node[clearbox]  (topleft) at (4.5,6) [label=above:$\phi_\omega(\pi(w_1 u))$] {};
\node[blackbox] (topright) at (13.5,6) [label=above:$\phi(\pi(w_2)h)$] {$\phi(\mathbb{P}[w_2])$};
\node[blackbox] (midleft) at (6.5,4)  {$\phi(1,\kappa(\pi(w_1)\cdot\mathbb{P}[u,v],\psi(w_2)*\eta))$};
\node[clearbox] (midright) at (15.5,4)  {};
\node [fill=white] at (15.5,4) { $\phi(1\cdot\kappa(\pi(w_1 u)h,\psi(w_2)*\eta),\kappa(\psi(w_2),\eta))$};
\node[clearbox] (bottomleft) at (4.5,2) [label=below:$\phi_\omega(\pi(w_1 v))$] {};
\node[blackbox] (bottomright) at (13.5,2) [label=below:$\phi(\pi(w_2)h)$] {$\phi(\mathbb{P}[w_2])^{-1}$};
\draw[very thick] (8,5) -- node[above]  {$\phi(1,\kappa(h,\psi(w_2)*\eta))$} (12,5);
\draw[very thick] (8,3) -- node[below]  {$\phi(1,\kappa(h,\psi(w_2)*\eta))$} (12,3);
\end{tikzpicture}
}
\end{center}
\caption{The path $\phi(\theta(\mathbb{E}))$ where $\mathbb{E} = (w_1, u=v, +1, w_2)$.}
\label{fig_phi_theta_E}
\end{figure}
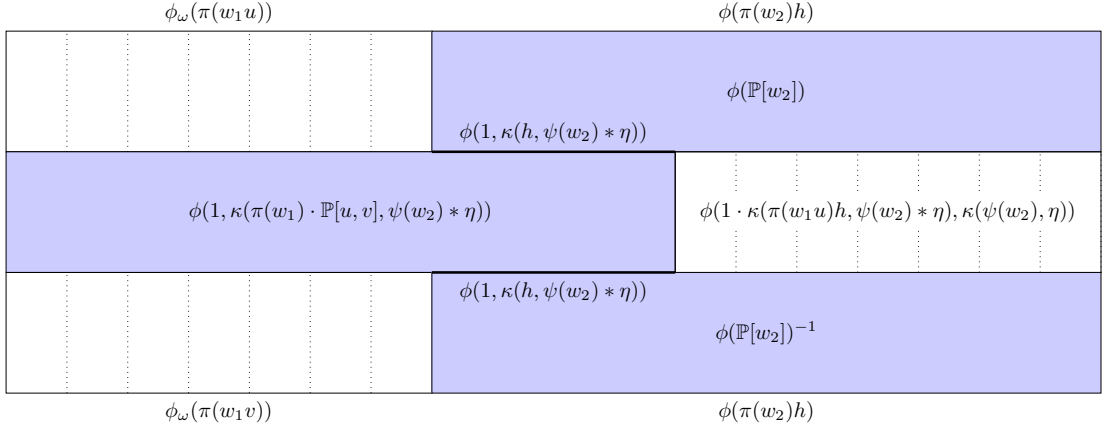
\begin{align*}
& \phantom{=} \;\;
\phi(\bbq_1) & & \\
& =
\phi(
\pi(w_1 u) \cdot \bbp[w_2]
) & & \\
& =
\phi(1, \kappa(\pi(w_1 u), \iota \bbp[w_2] * \eta)) \cdot
\phi(1 \cdot \kappa(\pi(w_1 u), \iota \bbp[w_2] * \eta), \kappa(\bbp[w_2], \eta))
& & \mbox{(by Corollary~\ref{corol_useful})}
\\
& =
\phi(1, \kappa(\pi(w_1 u), \pi(w_2)h * \eta)) \cdot
\phi(1 \cdot \kappa(\pi(w_1 u), \pi(w_2)h * \eta), \kappa(\bbp[w_2], \eta))
 & & \\
& =
\phi_\omega(\pi(w_1 u)) \cdot
\phi(1 \cdot \kappa(\pi(w_1 u), \omega), \kappa(\bbp[w_2], \eta))
 & & \mbox{(by Lemma~\ref{lem_5.10'}(i))} \\
& =
\phi_\omega(\pi(w_1 u)) \cdot
\phi(1, \kappa(\bbp[w_2], \eta))
 & & \mbox{(by Lemma~\ref{Ruskuc_Lemma5.4}(ii))} \\
& =
\phi_\omega(\pi(w_1 u)) \cdot
\phi(\bbp[w_2]). & &
\end{align*}
Similarly one obtains:
\[
\phi(\bbq_3)
=
\phi(
\pi(w_1 v) \cdot \bbp[w_2]^{-1}
)
=
\phi_\omega(\pi(w_1 v)) \cdot
\phi(\bbp[w_2])^{-1},
\]
and also applying Corollary~\ref{corol_useful}:
\begin{eqnarray*}
& & \phi(\bbq_2) \\
& = &
\phi(
\pi(w_1) \cdot \bbp[u,v] \cdot \psi(w_2)
)
\\
& = &
\phi(1, \kappa(\pi(w_1) \cdot \bbp[u,v], \psi(w_2) * \eta)) \cdot
\phi(1 \cdot \kappa( \pi(w_1) \cdot \iota\bbp[u,v], \psi(w_2) * \eta), \kappa(\psi(w_2), \eta))
\\
& = &
\phi(1, \kappa(\pi(w_1) \cdot \bbp[u,v], \psi(w_2) * \eta))
\cdot
\phi(1 \cdot \kappa(\pi(w_1 u) h , \psi(w_2) * \eta), \kappa(\psi(w_2), \eta)).
\end{eqnarray*}
Now
\begin{align*}
& \;
\iota \phi(1, \kappa(\pi(w_1) \cdot \bbp[u,v], \psi(w_2) * \eta)) & &
\\
& \equiv
\phi(1, \kappa(\pi(w_1) \cdot \iota \bbp[u,v], \psi(w_2) * \eta)) & &
\\
& \equiv
\phi(1, \kappa(\pi(w_1 u) h, \psi(w_2) * \eta)) & &
\\
& \equiv
\phi(1,
\kappa(\pi(w_1 u), h\psi(w_2) * \eta)
\kappa(h, \psi(w_2) * \eta) & & \mbox{(by Lemma \ref{Ruskuc_Lemma5.2}(i))}
\\
& \equiv
\phi(1,
\kappa(\pi(w_1 u), \omega)
\kappa(h, \psi(w_2) * \eta)
) & & \mbox{(since $\psi(w_2) \in H$)}
\\
& \equiv
\phi(1,
\kappa(\pi(w_1 u), \omega))
\phi(1 \cdot \kappa(\pi(w_1 u), \omega),
\kappa(h, \psi(w_2) * \eta)
) & & \mbox{(by Lemma \ref{Ruskuc_Lemma5.1}(i))}
\\
& \equiv
\phi_\omega(\pi(w_1 u))
\phi(1,
\kappa(h, \psi(w_2) * \eta)
). & & \mbox{(by Lemma~\ref{Ruskuc_Lemma5.4}(ii))}
\end{align*}
Similarly:
\[
\tau \phi(1, \kappa(\pi(w_1) \cdot \bbp[u,v], \psi(w_2) * \eta))
 \equiv
\phi_\omega(\pi(w_1 v))
\phi(1,
\kappa(h, \psi(w_2) * \eta).
\]
From this we see that for $\phi(\theta(\bbe))$ the hypotheses of Lemma~\ref{lem_form} hold (see Figures~\ref{fig_step0} and ~\ref{fig_phi_theta_E}) with:
\[
\begin{array}{rclcrcl}
\bbp & = & \phi(\bbp[w_2])
& &
\bbf & = & \phi(1, \kappa(\pi(w_1) \cdot \bbp[u,v], \psi(w_2) * \eta))
\\
\alpha & \equiv & \phi_\omega(\pi(w_1 u))
& &
\beta & \equiv & \phi_\omega(\pi(w_1 v))
\\
\gamma & \equiv & \phi(1 \cdot \kappa(\pi(w_1 u) h , \psi(w_2) * \eta), \kappa(\psi(w_2), \eta)
& &
y & \equiv & \phi(1,
\kappa(h, \psi(w_2) * \eta)
) \in W.
\end{array}
\]
Applying Lemma~\ref{lem_form} we conclude that $\phi(\theta(\bbe))$ is $\sim_{Y_2}$-homotopic to:
\begin{align*}
& \; \; \phantom{=}
[
(\alpha \cdot {\bbd_y}^{-1}) \circ
(\bbf \cdot \hat{y}) \circ
(\beta \cdot \bbd_y)
] \cdot \iota \phi(\bbp[w_2]).
\\
& = 
[
(\phi_\omega(\pi(w_1 u)) \cdot {\bbd_y}^{-1}) \circ
(\phi(1, \kappa(\pi(w_1) \cdot \bbp[u,v], \psi(w_2) * \eta)) \cdot \hat{y}) \circ
(\phi_\omega(\pi(w_1 v)) \cdot \bbd_y)
] \cdot \phi(\iota \bbp[w_2])
\\
& = 
\phi_\omega(\pi(w_1))
\cdot
\;
[
(\phi_\omega(\pi(u)) \cdot \bbd_y^{-1}) \circ
(\phi(1,\kappa(\bbp[u,v], \psi(w_2) * \eta )) \cdot \hat{y}) \circ
(\phi_\omega(\pi(v)) \cdot \bbd_y)
]
\;
\cdot
\phi(\pi(w_2)h).
\end{align*}
This completes the proof of the lemma.
\end{proof}



\begin{lem}
\label{lem_lambdagood}
Let $\bbe = w_1 \cdot \bba_{u=v} \cdot {w_2}$ be an arbitrary edge of $\Gamma(\gq)$, where $w_1, {w_2} \in B^*$ and $\bba_{u=v}$ is elementary. Then
\begin{align*}
\Lambda_{\iota \bbe}^{-1} \circ
\bbe \circ
\Lambda_{\tau \bbe}
& \sim_0 
(\tau \Lambda_{w_1}') \cdot
[{\Lambda_u'}^{-1} \circ \bba_{u=v} \circ \Lambda_v']
\cdot
(\tau \Lambda_{w_2}) \\
& = 
\phi_\omega(\pi(w_1)) \cdot
[{\Lambda_u'}^{-1} \circ \bba_{u=v} \circ \Lambda_v']
\cdot
\phi(\pi({w_2})h). \\
\end{align*}
\end{lem}
\begin{figure}
\begin{center}
\scalebox{0.7}
{
\begin{tikzpicture}
[blackbox/.style={draw, fill=blue!20, rectangle, minimum height=1.5cm, minimum width=1.5cm},
dottedbox/.style={draw, rectangle, loosely dashed, minimum height=3.5cm, minimum width=4cm},
clearbox/.style={draw, rectangle, minimum height=1.5cm, minimum width=1.5cm}]
\node at (1.75,10.75) [clearbox] {};
\node at (1.75+1.5,10.75) [clearbox] {};
\node at (1.75+1.5+1.5,10.75) [blackbox] {$\Lambda_\beta^{-1}$};
\node at (1.75,10.75-1.5) [clearbox] {};
\node at (1.75+1.5,10.75-1.5) [blackbox] {$\Lambda_u'^{-1}$};
\node at (1.75+1.5+1.5,10.75-1.5) [clearbox] {};
\node at (1.75,10.75-3) [blackbox] {$\Lambda_\alpha'^{-1}$};
\node at (1.75+1.5,10.75-3) [clearbox] {};
\node at (1.75+1.5+1.5,10.75-3) [clearbox] {};
\node at (1.75,10.75-4.5) [clearbox] {$\alpha \cdot$};
\node at (1.75+1.5,10.75-4.5) [blackbox] {$\mathbb{A}_{u=v}$};
\node at (1.75+1.5+1.5,10.75-4.5) [clearbox] {$\cdot \beta$};
\node at (1.75,10.75-6) [blackbox] {$\Lambda_\alpha'$};
\node at (1.75+1.5,10.75-6) [clearbox] {};
\node at (1.75+1.5+1.5,10.75-6) [clearbox] {};
\node at (1.75,10.75-7.5) [clearbox] {};
\node at (1.75+1.5,10.75-7.5) [blackbox] {$\Lambda_v'$};
\node at (1.75+1.5+1.5,10.75-7.5) [clearbox] {};
\node at (1.75,10.75-9) [clearbox] {};
\node at (1.75+1.5,10.75-9) [clearbox] {};
\node at (1.75+1.5+1.5,10.75-9) [blackbox] {$\Lambda_\beta$};
\node at (6.75,6.25) {$\sim_0$};
%
%
%
%
\node at (1.75+7,10.75) [blackbox] {$\Lambda_\alpha'^{-1}$};
\node at (1.75+1.5+7,10.75) [clearbox] {};
\node at (1.75+1.5+1.5+7,10.75) [clearbox] {};
\node at (1.75+7,10.75-1.5) [blackbox] {$\Lambda_\alpha'$};
\node at (1.75+1.5+7,10.75-1.5) [clearbox] {};
\node at (1.75+1.5+1.5+7,10.75-1.5) [clearbox] {};
\node at (1.75+7,10.75-3) [clearbox] {};
\node at (1.75+1.5+7,10.75-3) [blackbox] {$\Lambda_u'^{-1}$};
\node at (1.75+1.5+1.5+7,10.75-3) [clearbox] {};
\node at (1.75+7,10.75-4.5) [clearbox] {};
\node at (1.75+1.5+7,10.75-4.5) [blackbox] {$\mathbb{A}_{u=v}$};
\node at (1.75+1.5+1.5+7,10.75-4.5) [clearbox] {};
\node at (1.75+7,10.75-6) [clearbox] {};
\node at (1.75+1.5+7,10.75-6) [blackbox] {$\Lambda_v'$};
\node at (1.75+1.5+1.5+7,10.75-6) [clearbox] {};
\node at (1.75+7,10.75-7.5) [clearbox] {};
\node at (1.75+1.5+7,10.75-7.5) [clearbox] {};
\node at (1.75+1.5+1.5+7,10.75-7.5) [blackbox] {$\Lambda_\beta^{-1}$};
\node at (1.75+7,10.75-9) [clearbox] {};
\node at (1.75+1.5+7,10.75-9) [clearbox] {};
\node at (1.75+1.5+1.5+7,10.75-9) [blackbox] {$\Lambda_\beta$};
\end{tikzpicture}
}
\end{center}
\caption{Proof of Lemma~\ref{lem_lambdagood}.}
\label{fig_little_boxes}
\end{figure}
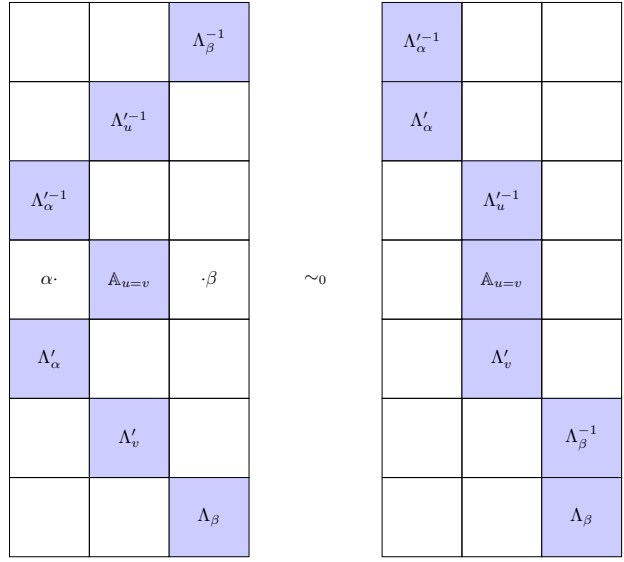
\begin{proof}
The proof is illustrated in Figure~\ref{fig_little_boxes}. 
We have:
\begin{align*}
&
\phantom{\sim_0}\;\;
\Lambda_{\iota \bbe}^{-1} \circ
\bbe \circ
\Lambda_{\tau \bbe} & & \\
& = 
({\Lambda_{w_1}'}^{-1} \ua
{\Lambda_u'}^{-1} \ua
{\Lambda_{w_2}}^{-1})
\circ
(w_1 \cdot \bba_{u=v} \cdot {w_2})
\circ
(\Lambda_{w_1}' \da \Lambda_v' \da \Lambda_{w_2}) &&  \mbox{(by definition of $\Lambda$)} \\
& \sim_0
({\Lambda_{w_1}'}^{-1} \circ \Lambda_{w_1}')
\da
({\Lambda_u'}^{-1} \circ \bba_{u=v} \circ \Lambda_v')
\da
({\Lambda_{w_2}}^{-1} \circ \Lambda_{w_2})  && \mbox{(by Lemma~\ref{lem_behaviour})}\\
& \sim_0
(\tau \Lambda_{w_1}') \cdot
[{\Lambda_u'}^{-1} \circ \bba_{u=v} \circ \Lambda_v']
\cdot
(\tau \Lambda_{w_2}). && \qedhere
\end{align*}
\end{proof}
Note that in both  Lemmas \ref{lem_reducing} and \ref{lem_lambdagood} the words acting on the left and right are, respectively, 
$\phi_\omega(\pi(w_1))$ and $\phi(\pi(w_2)h)$.
Therefore combining the two lemmas we immediately obtain the following.

\begin{lem}
\label{prop_thekey}
Let $\bbe = w_1 \cdot \bba_{u=v}^{\epsilon} \cdot {w_2}$ be an arbitrary edge of $\Gamma(\gq)$, where $w_1, {w_2} \in B^*$ and $\bba_{u=v}^{\epsilon}$ is elementary. Then in $\Gamma(\gq)$ the closed path
$\Lambda_{\iota \bbe}^{-1} \circ \bbe \circ \Lambda_{\tau \bbe} \circ \phi (\theta (\bbe))$ is $\sim_{Y_2}$-homotopic to
\begin{equation}
\label{eqn_bigone}
\phi_\omega(\pi(w_1))
\cdot 
\left[
{\Lambda_u'}^{-1} \circ \bba_{u=v}^{\epsilon} \circ \Lambda_v' \circ
(\phi_\omega(\pi(u)) \cdot \bbd_y^{-1}) \circ
\bbq_{u,v,\epsilon,w_2}
\circ
(\phi_\omega(\pi(v)) \cdot \bbd_y)
\right] 
\cdot \phi(\pi(w_2)h)
\end{equation}
where $y = \phi(1, \kappa(h, \psi(w_2) * \eta) \in W$ and
$\bbq_{u,v,\epsilon,w_2} = \phi(1,\kappa(\bbp[u,v]^\epsilon, \psi(w_2) * \eta )) \cdot \hat{y}$.
\end{lem}
\begin{proof}
This follows immediately from Lemmas \ref{lem_reducing} and \ref{lem_lambdagood}. 
\end{proof}
This brings us to the main result of this section.
The key observation is that the number of possibilities for the middle term in \eqref{eqn_bigone}  is easily seen to be finite when each of the sets $J$, $\Lambda$, $A$ and $\mathfrak{R}$ is finite. In particular, even though $w_2 \in B^*$ ranges over infinitely many possible words, there are still only finitely many possibilities for $\psi(w_2) * \eta$ (since $J$ is assumed to be finite), and consequently there are only finitely many possibilities for $\bbq_{u,v,\epsilon,w_2}$. If in addition the homotopy base $X$ for $S$ is finite, then, we obtain a finite homotopy base for the Sch\"{u}tzenberger group $\gG$.
\begin{thm}
\label{thm_baseforG}
Let $S$ be a monoid defined by a presentation $\lb A | \mathfrak{R} \rb$, $H$ be an $\gh$-class of $S$, and let $\gG$ be the Sch\"{u}tzenberger group of $H$.

Then, with the above notation, $\gG$ is defined by the presentation $\lb B | \mathfrak{U} \rb$. Furthermore, if $X$ is a homotopy base for $\lb A | \mathfrak{R} \rb$ then $Y_1 \cup Y_2 \cup Y_3$ is a homotopy base for $\lb B | \mathfrak{U} \rb$ where:
\begin{align*}
Y_1 & = 
\{ \phi(\lambda \cdot \kappa(\alpha, \iota \bbp \beta * j), \kappa(\bbp, \beta * j))
\;:\;
\bbp \in X, \;
\alpha, \beta \in A^*, \;
j \in J, \;
\lambda \in \Lambda, \;
H_{\lambda} \subseteq Sr_{{\alpha \iota \bbp \beta} * j}
\}, \\
Y_2 & =  \{
(\bbd_y^{-1} \cdot y) \circ (y \cdot \bbd_y^*):
y \in W
\},
\end{align*}
with $W = \{
\phi (1, \kappa(h,j)): j \in J, h * j \neq 0
\}$, and $Y_3$ is the set of all paths:
\[
{\Lambda_u'}^{-1} \circ \bba_{u=v}^{\epsilon} \circ \Lambda_v' \circ
(\phi_\omega(\pi(u)) \cdot \bbd_y^{-1}) \circ
\bbq_{u,v,\epsilon,w_2}
\circ
(\phi_\omega(\pi(v)) \cdot \bbd_y)
\]
where 
$w_2 \in A^*$, 
$(u=v) \in \mathfrak{U}$, 
$\epsilon=\pm 1$,
$y = \phi(1, \kappa(h, \psi(w_2) * \eta)$ and
\[
\bbq_{u,v,\epsilon,w_2} = \phi(1,\kappa(\bbp[u,v]^\epsilon, \psi(w_2) * \eta )) \cdot \hat{y}.
\]
In particular, if $A$, $\mathfrak{R}$, $J$, $\Lambda$ and $X$ are all finite then both the presentation $\lb B | \mathfrak{U} \rb$ and the homotopy base $Y_1 \cup Y_2 \cup Y_3$ are finite.
\end{thm}
\begin{proof}
That $\lb B | \mathfrak{U} \rb$ presents $\gG$ was proved in Theorem~\ref{thm_thepresentation}.

Now let $X$ be a homotopy base for $\lb A | \mathfrak{R} \rb$.
By Lemma~\ref{lem_infinitetriv}, $Z \cup Y_1$ is a homotopy base for  $\lb B | \mathfrak{U} \rb$. But it follows from Lemma~\ref{prop_thekey} that
$
\sim_{Y_2 \cup Y_3} \supseteq \sim_{Z}.
$
Thus $\sim_{Y_1 \cup Y_2 \cup Y_3} \supseteq \sim_{Y_1 \cup Z}$, and we conclude that $Y_1 \cup Y_2 \cup Y_3$ is a homotopy base for  $\lb B | \mathfrak{U} \rb$.
\end{proof}
Our main result, Theorem~\ref{thm_main_result} stated in Section~1 above, is an immediate consequence of Theorems~\ref{thm_thepresentation} and ~\ref{thm_baseforG}.

\begin{proof}[Proof of Theorem~\ref{thm_main_result}]
It follows from assumption (a) that the index set $\Lambda$ is finite, and from assumption (b) that the index set $J$ is finite. If $S$ is finitely presented then by Theorem~\ref{thm_thepresentation} so is the Sch\"{u}tzenberger group $\mathcal{G}$, proving (i). If in addition $S$ has finite derivation type then $A$, $\mathfrak{R}$, $J$, $\Lambda$ and $X$ are all finite and so by Theorem~\ref{thm_baseforG}, $Y_1 \cup Y_2 \cup Y_3$ is a finite homotopy base for the Sch\"{u}tzenberger group $\mathcal{G}$ and so $\mathcal{G}$ has finite derivation type, proving (ii). 
\end{proof}

\bibliographystyle{abbrv}

\def\cprime{$'$} \def\cprime{$'$} \def\cprime{$'$}

\end{document}